\documentclass[11pt]{article}

\usepackage{latexsym}
\usepackage{amssymb}
\usepackage{amsthm}
\usepackage{amscd}

\usepackage{amsmath}

\theoremstyle{definition}
\newenvironment{conjecture}[2][Conjecture]{\begin{trivlist}
\item[\hskip \labelsep {\bfseries #1}\hskip \labelsep {\bfseries \hspace{-0.5mm}#2.}]}{\end{trivlist}}

\newcommand{\descref}[1]{\hyperref[#1]{#1}}

\usepackage{mathrsfs}
\usepackage[all]{xy}
\usepackage{hyperref} 
\usepackage[usenames,dvipsnames]{color}
\usepackage{graphicx,eepic}
\usepackage{float}

\usepackage{color}

\theoremstyle{definition}
\newtheorem* {theorem*}{Theorem}
\newtheorem{theorem}{Theorem}[section]
\newtheorem{thmdef}[theorem]{Theorem-Definition}

\theoremstyle{definition}
\newtheorem{observation}[theorem]{Observation}

\newtheorem* {example*}{Example}

\newtheorem{lemma}[theorem]{Lemma}
\theoremstyle{definition}
\newtheorem{definition}[theorem]{Definition}
\theoremstyle{definition}
\newtheorem* {notation}{Notation}
\newtheorem{proposition}[theorem]{Proposition}
\newtheorem{corollary}[theorem]{Corollary}

\newtheorem* {remark}{Remark}
\theoremstyle{definition}

\theoremstyle{definition}

\theoremstyle{definition}

\theoremstyle{definition}

\xyoption{dvips}

\usepackage{fullpage}

\numberwithin{equation}{section}

\def\modu{\ (\mathrm{mod}\ }

\def\({\left(}
\def\){\right)}
       \newcommand{\QQ}{\mathbb{Q}}    \newcommand{\cA}{\mathcal{A}}

\def\NN{\mathbb{N}}
    \def\ZZ{\mathbb{Z}} \def\Aut{\mathrm{Aut}}  
         \def\spanning{\textnormal{-span}}   
  \def\wt{\widetilde}

\newcommand{\cM}{\mathcal{M}}

\def\barr{\begin{array}}
\def\earr{\end{array}}
\def\ba{\begin{aligned}}
\def\ea{\end{aligned}}
\def\be{\begin{equation}}
\def\ee{\end{equation}}

\def\qquand{\qquad\text{and}\qquad}
\def\quand{\quad\text{and}\quad}

\def\I{\mathbf{I}_*}
\def\M{\mathcal{M}}
\def\Des{\mathrm{Des}_L}
\def\omdef{\overset{\mathrm{def}}}
\def\uS{\underline S}

\def\hs{\hspace{0.5mm}}

\def\act{\ltimes}
\newcommand{\Psig}{P^\sigma}

\def\proj{\mathrm{proj}}

\def\cH{\mathcal H}
\def\cM{\mathcal M}

\def\M{\mathcal M_{q^2}}

\def\A{A}
\def\a{a}

\def\ben{\begin{enumerate}}
\def\een{\end{enumerate}}

\makeatletter
\renewcommand{\@makefnmark}{\mbox{\textsuperscript{}}}
\makeatother

\allowdisplaybreaks[1]

\UseCrayolaColors

\begin{document}
\title{Positivity conjectures for Kazhdan-Lusztig theory on twisted involutions: the universal case}
\author{Eric Marberg
 \\ Department of Mathematics \\ Massachusetts Institute of Technology \\ \tt{emarberg@math.mit.edu}
}
\date{}

\maketitle

\begin{abstract}
Let $(W,S)$ be a Coxeter system and let $w \mapsto w^*$ be an involution of $W$ which preserves the set of simple generators $S$. Lusztig and Vogan have recently shown that the set of twisted involutions (i.e., elements $w \in W$ with $w^{-1} = w^*$) naturally generates a module of the Hecke algebra of $(W,S)$ with two  distinguished bases. The  transition matrix between these bases defines  a family of polynomials $P^\sigma_{y,w}$ which one can view as  ``twisted'' analogues of the much-studied Kazhdan-Lusztig polynomials of $(W,S)$. The polynomials  $P^\sigma_{y,w}$ can have negative coefficients, but display several conjectural positivity properties of interest. This paper reviews Lusztig's construction and then proves three such positivity properties for Coxeter systems which are universal (i.e., having no braids relations), generalizing previous work of Dyer. Our methods are entirely combinatorial and elementary, in contrast to the geometric arguments employed by Lusztig and Vogan to prove similar positivity conjectures for crystallographic Coxeter systems.
\end{abstract}

\setcounter{tocdepth}{2}
\tableofcontents

\section{Introduction}

\subsection{Overview} \label{overview-sect}

A nice source of open problems in the representation theory of Coxeter systems  comes from the frequent observation that interesting properties of Weyl groups 
 seem to hold for much larger
classes of reflection groups.
This paper concerns phenomena of this nature which have arisen in ongoing work of Lusztig and Vogan \cite{LV2,LV5,LV,LV3,LV4}.


 Let $(W,S)$ be any Coxeter system, and write $\cH_{q^2}$ for the associated \emph{generic Hecke algebra with parameter $q^2$}: this  is  the usual Hecke algebra (namely, a certain $\ZZ[q^{\pm 1/2}]$-algebra with a basis $\(T_w\)_{w\in W}$ indexed by $W$), but with $q$ replaced by $q^2$ in its defining relations.  A precise definition appears in Section \ref{twisted-intro} below.
%
%
%
%
Next, fix
an automorphism  $* : W \to W$ with order one or two which preserves the set of simple generators   $S$. Write $\I$ for the corresponding set of twisted involutions (i.e., elements $w \in W$ with $w^{-1} = w^*$), 
 and let 
$\cM_{q^2}$ be the free $\ZZ[q^{\pm 1/2}]$-module which this set generates.

 Lusztig \cite{LV2} has shown that   $\cM_{q^2}$ 
 has an  $\cH_{q^2}$-module structure which serves as a natural and interesting analogue of the regular representation of $\cH_{q^2}$ on itself. (Section \ref{twisted-intro} contains the details of this construction.)
The regular representation of $\cH_{q^2}$ possesses a distinguished
  \emph{Kazhdan-Lusztig basis} $( C_w )_{w \in W}$, whose transition matrix from the standard basis $( T_w )_{w \in W}$ defines the much-studied family of \emph{Kazhdan-Lusztig polynomials} $(P_{y,w})_{y,w \in W} \subset \ZZ[q]$. 
  Lusztig's work \cite{LV2} indicates that we may repeat much of this theory for the module $\cM_{q^2}$: it too has a ``Kazhdan-Lusztig basis'' whose transition matrix from the standard basis  defines a family of ``twisted Kazhdan-Lusztig polynomials'' $(P^\sigma_{y,w})_{y,w \in \I} \subset \ZZ[q]$.  

 Many remarkable properties of the Kazhdan-Lusztig basis of $\cH_{q^2}$ appear to have ``twisted'' analogues for the module $\cM_{q^2}$.
For example, one of the most famous aspects of the original Kazhdan-Lusztig polynomials $(P_{y,w})_{y,w \in W}$ is that their   coefficients are always  nonnegative. (This statement, while known in many cases from the work of a number of people, has only recently been proved for all Coxeter systems by Elias and Williamson \cite{EWpaper}.) 
The twisted Kazhdan-Lusztig polynomials $(P^\sigma_{y,w})_{y,w \in \I}$ can have negative coefficients. However, Lusztig and Vogan \cite{LV}   have shown by  geometric arguments   that   the modified polynomials $ \frac{1}{2} (P_{y,w} \pm  P^\sigma_{y,w})$ for $y,w \in \I$  have nonnegative  coefficients whenever $W$ is crystallographic. In fact, for any choice of $(W,S)$ and $*$,  the polynomials $\frac{1}{2} (P_{y,w} \pm  P^\sigma_{y,w})$ belong to $\ZZ[q]$, and Lusztig \cite{LV2} has conjectured that their coefficients are  always  nonnegative. 

Section \ref{conjecture-sect} presents two other positivity conjectures for the ``Kazhdan-Lusztig basis'' of the twisted involution $\cH_{q^2}$-module $\cM_{q^2}$. These serves as  analogues of  longstanding  conjectures  related to the ordinary Kazhdan-Lusztig polynomials $P_{y,w}$. 
After stating 
 these ``twisted'' conjectures, we devote the rest of this paper to proving them for Coxeter systems  $(W,S)$ which are \emph{universal} (i.e., such that $st \in W$ has infinite order for all distinct $s, t \in S$), with $*$ arbitrary.  
 This special case is of interest as it 
 provides an infinite family of Coxeter systems for which our conjectures hold, despite existing outside the geometric context of Weyl groups and affine Weyl groups. 
 Moreover, it is possible in the universal case to derive explicit formulas for the polynomials $P_{y,w}$ and $\Psig_{y,w}$.
 
These results generalize Dyer's work \cite{Dyer} on the Kazhdan-Lusztig polynomials of universal Coxeter systems. A  detailed summary of our methods  appears in Section \ref{summary-sect} at the end of this introduction. (Sections \ref{setup-sect}$-$\ref{conjecture-sect} provide some brief preliminaries needed to state our main results.) The study of the module $\cM_{q^2}$ and its conjectural positivity properties  continues in our companion paper \cite{EM2}, which addresses the case when $(W,S)$ is  finite.

\subsection{Setup}\label{setup-sect}

Throughout we write $\ZZ$ for the integers and $\NN = \{0,1,2,\dots\}$ for the nonnegative integers. We also adopt the following conventions:

\begin{itemize}
\item Let $(W,S)$ be a Coxeter system with length function $\ell : W \to \NN$.


\item Let $\leq $ denote the Bruhat order on $W$. Recall that in this partial order we have $y \leq w$  if and only if for each reduced expression $w = s_1\cdots s_k$ with $s_i \in S$, we have  $y = s_{i_1}\cdots s_{i_m}$ for some  sequence of indices $1\leq i_1 < \dots < i_m \leq k$.


\item Let $\cA = \ZZ[v,v^{-1}]$ be the ring of Laurent polynomials over $\ZZ$ in an indeterminate $v$.

\item Let $q =v^2$. In the sequel, we will refer to $v$ in place of the parameter $q^{1/2}$ used in Section \ref{overview-sect}.

\end{itemize}
The ring $\cA$ will now occupy the role which $\ZZ[q^{\pm1/2}]$ played in the previous section. For background on Coxeter systems and the Bruhat order, see for example \cite{CCG,Hu,Lu}.

\subsection{Kazhdan-Lusztig theory}\label{kl-intro}

Here we briefly recall the definition of the Kazhdan-Lusztig polynomials attached to $(W,S)$.
Let $\cH_q$ denote the free $\cA$-module with basis $\{ t_w : w \in W\}$. 
This module has a unique $\cA$-algebra structure  with respect to which the multiplication rule
\[ t_s t_w = \begin{cases} t_{sw} & \text{if }\ell(sw) = \ell(w) +1 \\ q t_{sw} + (q-1) t_w &\text{if }\ell(sw) = \ell(w)-1 \end{cases}
\] holds 
for each $s \in S$ and $w \in W$. 
 The element $t_w \in \cH_q$ is  more often denoted in the literature by the symbol $T_w$, but here we reserve the latter notation for the   Hecke algebra $\cH_{q^2}$, to be introduced in the next section.

We refer to the algebra $\cH_q$ as the \emph{Hecke algebra of $(W,S)$ with parameter $q$.} A number of good references exist for this much-studied object; see for example \cite{CCG,Hu,KL,Lu}.
The Hecke algebra possesses a unique ring involution  $\overline{\ } : \cH_q \to \cH_q$ with 
$\overline{v^n} = v^{-n}$ and $\overline {t_w} = \(t_{w^{-1}}\)^{-1}$ 
for all $n \in \ZZ$ and $w \in W$, referred to as the \emph{bar operator}, and this gives rise to the following  theorem-definition from Kazhdan and Lusztig's seminal paper \cite{KL}.

\begin{thmdef}[Kazhdan and Lusztig \cite{KL}]  \label{kl-thmdef} For each $w \in W$ there is a unique family of polynomials $\( P_{y,w} \)_{y \in W} \subset \ZZ[q]$ 
with the following three properties:
\ben
\item[(a)] The element $ c_w \omdef =  v^{-\ell(w)} \cdot \sum_{y \in W}   P_{y,w}\cdot t_y $ in $ \cH_q$
has $\overline{c_w} = c_w$.
\item[(b)] $P_{y,w} = \delta_{y,w}$ if $y \not < w$ in the Bruhat order.
\item[(c)] $P_{y,w}$ has degree at most $ \frac{1}{2} \( \ell(w)-\ell(y)-1\)$  as a polynomial in $q$ whenever $y < w$.
\een
\end{thmdef}

\begin{remark} Here and elsewhere,  the Kronecker delta $\delta_{y,w}$ has the usual meaning of   $\delta_{y,w} =1$ if $y=w$ and $\delta_{y,w} = 0$ otherwise. 
\end{remark}

The polynomials $(P_{y,w})_{y,w \in W}$ are the \emph{Kazhdan-Lusztig polynomials} of the Coxeter system $(W,S)$. Property (b) implies that the  
elements $\( c_w \)_{ w \in W}$ form an $\cA$-basis for  $\cH_q$, which one calls the \emph{Kazhdan-Lusztig basis}. 
For more information on the Kazhdan-Lusztig polynomials and methods of computing them, see, for example,  \cite[Chapter 7]{Hu} or \cite[Chapter 5]{CCG}.

\subsection{``Twisted Kazhdan-Lusztig theory''}\label{twisted-intro}

We now present Lusztig's definition of the module $\cM_{q^2}$ and the polynomials $\Psig_{y,w}$  mentioned at the start of this introduction.
To begin, 
 %
we let $\cH_{q^2}$ denote the \emph{Hecke algebra of $(W,S)$ with parameter $q^2$}: this is the free $\cA$-module with basis $\{ T_w : w \in W\}$, equipped with the unique $\cA$-algebra structure  with respect to which the multiplication rule
\[ T_s T_w = \begin{cases} T_{sw} & \text{if }\ell(sw) = \ell(w) +1 \\ q^2 T_{sw} + (q^2-1) T_w &\text{if }\ell(sw) = \ell(w)-1 \end{cases}
\] holds 
for each $s \in S$ and $w \in W$.
Like $\cH_q$, this algebra  possesses a unique ring involution  $\overline{\ } : \cH_{q^2} \to \cH_{q^2}$ with 
$\overline{v^n} = v^{-n}$ and $\overline {T_w} = \(T_{w^{-1}}\)^{-1}$
for all $n \in \ZZ$ and $w \in W$. This bar operator fixes each of the elements
 \[C_w \omdef = q^{-\ell(w)} \cdot \sum_{y \in W} P_{y,w}(q^2) \cdot T_y\qquad\text{for }w \in W.\] The elements  $\( C_w \)_{ w \in W}$ form an $\cA$-basis of $\cH_{q^2}$ which one refers to as the \emph{Kazhdan-Lusztig basis}.
The use of the capitalized symbols $T_w$, $C_w$ is intended to distinguish elements of $\cH_{q^2}$ from the basis elements $t_w$, $c_w$ of the usual Hecke algebra $\cH_q$.

The following Theorem-Definition of Lusztig \cite{LV2}  defines 
 $\cM_{q^2}$ explicitly as a certain module of the algebra $\cH_{q^2}$.
This statement requires a few additional ingredients:
\begin{itemize}

\item Fix an automorphism $w \mapsto w^*$ of $W$ with order $\leq 2$ such that $s^* \in S$ for each $s \in S$. 

\item Set $\I = \{ w \in W : w^* = w^{-1}\}$. One calls elements of this set \emph{twisted involutions}.



\item Given $s \in S$ and $w \in \I$, let $s \act w$ denote the unique element in the intersection of $\{ sw,sws^*\}$ and $\I\setminus \{w\}$. Note that while $s \act (s \act w) = w$, the operation $\act : S \times \I\to \I$ generally does not extend to a group action of $W$ on $\I$.

\end{itemize}
We now have Lusztig's result.  
This statement first appeared in Lusztig and Vogan's paper \cite{LV} in the special case that $W$ is a Weyl group or affine Weyl group and $*$ is trivial. 

\begin{thmdef}[Lusztig and Vogan \cite{LV}; Lusztig \cite{LV2}]\label{lv-thmdef}
Let $\M$ be the free $\cA$-module with basis $\{ \a_w : w \in \I\}$.  
\ben
\item[(a)]  $\M$ has a unique $\cH_{q^2}$-module structure with respect to which the following multiplication rule
holds
for each  $s \in S$ and $w \in \I$:
\be\label{module-def} T_s \a_w =
\left\{
\ba 
a_{s\act w}	\ {\color{white}+}\ &			&&\qquad\text{if $s\act w =sws^* > w$} \\
(q+1)a_{s\act w} \ +\ & qa_w			&&\qquad\text{if $s\act w =sw > w$} \\
(q^2-q)a_{s\act w} \ +\ &(q^2-q-1) a_w 	&&\qquad\text{if $s\act w =sw < w$} \\
q^2a_{s\act w} \ +\  & (q^2-1)a_w 		&&\qquad\text{if $s\act w =sws^* < w$.}
\ea
\right.
\ee

\item[(b)] There is a unique $\ZZ$-linear involution $\overline{\ } : \M \to \M$ such that $\overline{a_1} = a_1$ and $\overline{h\cdot m} = \overline h \cdot \overline m$ for all $h \in \cH_{q^2}$ and $m \in \cM_{q^2}$.
This bar operator acts on the standard basis of $\cM_{q^2}$ by the  formula $\overline{\a_w} =(-1)^{\ell(w)}\cdot  (T_{w^{-1}})^{-1}\cdot \a_{w^{-1}}$ for $w \in \I$. 

\een
\end{thmdef}

The bar operator just introduced   on  $\cM_{q^2}$  gives rise, in turn, to the  following analogue of Theorem-Definition \ref{kl-thmdef}. Like the previous result, this was first shown by Lusztig and Vogan \cite{LV} in the crystallographic case (with $*$ trivial). Lusztig \cite{LV2} subsequently extended the statement to all Coxeter systems.

\begin{thmdef}[Lusztig and Vogan \cite{LV}; Lusztig \cite{LV2}] \label{twisted-thmdef} For each $w \in \I$ there is a unique family of polynomials $\( \Psig_{y,w} \)_{y \in \I} \subset \ZZ[q]$ 
with the following three properties:
\ben
\item[(a)] The element $ 
A_w \omdef=  v^{-\ell(w)} \cdot \sum_{y \in \I}   \Psig_{y,w}\cdot a_y$ in $ \M$
has $\overline{A_w} = A_w$.
\item[(b)] $\Psig_{y,w} = \delta_{y,w}$ if $y \not < w$ in the Bruhat order.
\item[(c)] $\Psig_{y,w}$ has degree at most $ \frac{1}{2} \( \ell(w)-\ell(y)-1\)$  as a polynomial in $q$ whenever $y < w$.
\een
\end{thmdef}

Note from (b) that the elements $\( A_w \)_{ w \in \I}$ form an $\cA$-basis for the module $\cM_{q^2}$.
We sometimes refer to this as the ``twisted Kazhdan-Lusztig basis.'' 
Likewise, we call the polynomials $\Psig_{y,w}$ the \emph{twisted Kazhdan-Lusztig polynomials} of the triple $(W,S,*)$. 
We will discuss some general properties of these polynomials  (and also address  how one computes them) in Section \ref{algo-sect} below. 

Before continuing to state the conjectures concerning $\Psig_{y,w}$ which are our main subject, let us mention a few reasons why one might care about these polynomials or the module $\cM_{q^2}$.
First, as detailed in \cite{LV}, when $W$ is a Weyl group or affine Weyl group, the module $\cM_{q^2}$  arises from geometric considerations and in that context the polynomials $\Psig_{y,w}$ are  expected to have importance in the theory of unitary representations of complex reductive groups.

While for more general Coxeter groups we lack such an interpretation for $\cM_{q^2}$,  
 there is nevertheless always a sense in which we can view the left regular representation of the Hecke algebra of a Coxeter system as 
a special case of (a submodule of)
the module $\cM_{q^2}$. Consequently,
one can realize the ordinary Kazhdan-Lusztig polynomials of one Coxeter system as the twisted polynomials $\Psig_{y,w}$ 
corresponding to another Coxeter system with a particular choice of $*$.
These considerations are explained in more precise detail in \cite{EM2}.

We also mention that when $W$ is finite, the irreducible decomposition of $\cM_{q^2}$ has a surprising interpretation in terms of the ``Fourier transform'' of a set of ``unipotent characters'' attached to $(W,S)$. This phenomenon, which is studied in various cases in the articles \cite{Ca,GM,Kottwitz,LV,EM}, gives  one more indication that  $\cM_{q^2}$ deserves consideration not only in the crystallographic case.

\subsection{Positivity conjectures}\label{conjecture-sect}

Many results in the theory of Hecke algebras depend on positivity properties of the Kazhdan-Lusztig polynomials $P_{y,w}$. In particular, we recall the following much studied conjectures:

\begin{conjecture}{A}\label{A}
The polynomials $P_{y,w}$ have nonnegative integer coefficients. 
\end{conjecture}

\begin{conjecture}{B}\label{B}
 The   polynomials $P_{y,w}$ are decreasing for fixed $w$, in the sense that the difference $P_{y,w} - P_{z,w}$ has nonnegative integer coefficients whenever $y \leq z$.
\end{conjecture}

Denote the structure coefficients of $\cH_q$ in the Kazhdan-Lusztig basis by $\( h_{x,y;z} \)_{x,y,z \in W}$; i.e., these are the Laurent polynomials in $ \cA$ satisfying 
$c_x c_y = \sum_{z \in W} h_{x,y;z} c_z $ for $x,y,z \in W$.

\begin{conjecture}{C}\label{C}
The Laurent polynomials $h_{x,y;z}$ have nonnegative coefficients. 
\end{conjecture}

These conjectures have been proved in the case when $(W,S)$ is crystallographic (i.e., when $W$ a Weyl group or affine Weyl group), finite, or universal through the work of a number of people \cite{Alvis,Fokko,Dyer,Ir,KL2,Laff,Springer}. 
Elias and Williamson's recent proof of Soergel's conjecture \cite{EWpaper}, finally, establishes  Conjectures \descref{A} and \descref{C}  for any Coxeter system. In this generality Conjecture \descref{B} remains open.

The central topic of this work concerns ``twisted'' versions of the preceding conjectures. 
While the parallels between Theorem-Definitions \ref{kl-thmdef} and \ref{twisted-thmdef} suggest  obvious analogues of Conjectures \descref{A}, \descref{B}, and \descref{C} in the twisted case,  these statements turn out not  to be the right ones.
Notably, the polynomials $\Psig_{y,w}$ may have negative coefficients.  To state the ``correct'' conjectures, define $P^+_{y,w}, P^-_{y,w} \in \QQ[q]$ by 
\be \label{pm-def} P^\pm_{y,w} = \tfrac{1}{2} \( P_{y,w} \pm \Psig_{y,w}\)\qquad\text{for each }y,w \in \I.\ee Lusztig proves that these polynomials actually have integer coefficients  \cite[Theorem 9.10]{LV2} and  conjectures the following:

\begin{conjecture}{A$'$}\label{A$'$}
 The   polynomials $P_{y,w}^+$ and $P^-_{y,w}$ have nonnegative integer coefficients. 
\end{conjecture}

This statement is a refinement of Conjecture \descref{A} since $P^+_{y,w} + P^-_{y,w} = P_{y,w}$ for $y,w \in \I$.
We  introduce the following stronger conjecture, which is likewise a refinement of Conjecture \descref{B}.

\begin{conjecture}{B$'$}\label{B$'$} The   polynomials $P^\pm_{y,w}$ are decreasing for fixed $w$, in the sense that the differences
$P^+_{y,w} - P^+_{z,w}$ and  $P^-_{y,w} - P^-_{z,w}$ have nonnegative integer coefficients whenever $y \leq z$.
\end{conjecture}

Finally, to provide an analog of Conjecture \descref{C}, for each $x \in W$ and $y \in \I$ define $\bigl( \wt h_{x,y;z}\bigr)_{z \in W}$ and $ \bigl(h^\sigma_{x,y;z}\bigr)_{z \in \I} $ as the Laurent polynomials in $\cA$ satisfying
\be\label{h-def} c_x c_y c_{{x^*}^{-1}} = \sum_{z \in W} \wt h_{x,y;z} c_z
\qquand
C_x \A_y = \sum_{z \in\I} h^\sigma_{x,y;z}\A_z.
\ee
Note  that $c_x,c_y,c_z \in \cH_q$  while $C_x \in \cH_{q^2}$ and $A_y \in \cM_{q^2}$. Now, 
 define $h^+_{x,y;z}, h^-_{x,y;z} \in \QQ[v,v^{-1}]$ by \be\label{hpm-def}
 h^\pm_{x,y;z}  = \tfrac{1}{2}\( \wt h_{x,y;z} \pm  h^\sigma_{x,y;z}\)
  \qquad\text{for each $x \in W$ and $y,z \in \I$.}
 \ee One can show from results of Lusztig \cite{LV2} that these Laurent polynomials likewise have integer coefficients (see Proposition \ref{h-parity} below), which leads to this conjecture.

\begin{conjecture}{C$'$}\label{C$'$}
The Laurent polynomials $h^+_{x,y;z}$ and $h^-_{x,y;z}$ have nonnegative integer coefficients.
\end{conjecture}


Lusztig and Vogan's work \cite{LV} establishes Conjecture \descref{A$'$} when $W$ is a Weyl group or affine Weyl group. In these cases, \cite[Section 5]{LV} also 
mentions without proof that Conjecture \descref{C$'$}  holds (when $*$ is trivial). Conjecture \descref{B$'$} appears still to be open even in the crystallographic case. Here, we will provide some  evidence that these conjectures  hold for all Coxeter systems, by proving them in the case that $W$ is universal. The supplementary paper  \cite{EM2} will provide further evidence coming from the case of finite Coxeter systems.


\subsection{Outline of main results}
\label{summary-sect}

Following Dyer \cite{Dyer}, we say that a Coxeter system $(W,S)$ is \emph{universal} if
the product $st \in W$ has infinite order
for any distinct generators $s,t \in S$. 
In this case $W$ is the group generated by $S$ subject only to the relations $s^2=1$ for $s \in S$.
The elements of $W$ consists of all words in $S$ with distinct adjacent letters, and products of  elements are given by concatenation, subject to the rule that one inductively removes all pairs of equal adjacent letters.


Let $(W,S)$ be any universal Coxeter system and   let $* \in \Aut(W)$ be any $S$-preserving involution of $W$. Restricted to $S$, the map $w \mapsto w^*$  then corresponds to either the identity or to an arbitrary permutation of order two. 
Dyer's paper \cite{Dyer} derives formulas for the polynomials $P_{y,w}$ and for the decomposition of the products $c_xc_y \in \cH_q$ in terms of  the Kazhdan-Lusztig basis, thus establishing Conjectures \descref{A}, \descref{B}, and \descref{C} in the universal case. (Dyer's results are formulated in somewhat different language than these conjectures; cf. Theorems \ref{dyer-thm1} and \ref{c-structure} below.)
Our paper proceeds as something of a sequel to Dyer's work, as follows:
\begin{itemize}
\item In Sections \ref{3a-sect}, \ref{3b-sect}, and \ref{technical-sect} we derive a series of  recurrence relations, with coefficients in $\NN[q]$, for the polynomials $\Psig_{y,z;w} \omdef = \Psig_{y,w} - \Psig_{z,w}$ and $P_{y,z;w} \omdef = P_{y,w} - P_{z,w}$ (with $y \leq z$).

\item These recurrences show that in the universal case  $\Psig_{y,w}$ and $P^\pm_{y,w}$   belong to $\NN[q]$ and are decreasing with respect to the index $y \in \I$ and the Bruhat order; see Theorems \ref{atypical-thm} and \ref{typical-thm} below. 
Thus  Conjectures \descref{A$'$} and \descref{B$'$} hold for universal Coxeter systems.
 
 \item In Section \ref{structure-sect} we describe the decomposition of the product $C_x A_y$ in terms of the distinguished basis $\( A_z\)_{z \in \I}$ of $\cM_{q^2}$; see Theorem \ref{A-structure}. This shows that the Laurent polynomials $h^\sigma_{x,y,;z}$ have nonnegative coefficients in the universal case; see Corollary \ref{A-cor}.
 
 \item Combining these results with Dyer's work finally affords a proof of Conjecture \descref{C$'$} for universal Coxeter systems; see Theorem \ref{last-thm}.
 
\end{itemize}
  Before carrying all this out, we provide in Section \ref{2sect}  a few relevant preliminaries concerning the Bruhat order on $\I$, the polynomials $P_{y,w}$ and $\Psig_{y,w}$, and the associated bases of $\cH_q$ and $\cM_{q^2}$. 

\subsection*{Acknowledgements}

I am grateful  to George Lusztig and David Vogan for helpful discussions and suggestions.

\section{Preliminaries}\label{2sect}


Here, we preserve all conventions from the introduction. Thus, $(W,S)$ is an arbitrary Coxeter system (not necessarily universal) with an 
 $S$-preserving involution $* \in \Aut(W)$, and attached to these choices are the following structures:
 \begin{itemize}
 \item  $\cH_{q^2} = \cA\spanning\{ T_w : w \in W\}$ is the  Hecke algebra of $(W,S)$ with parameter $q^2$.
 
 \item $\I = \{ w \in W : w^{-1} = w^*\}$ is the corresponding set of twisted involutions.
 
 \item $\cM_{q^2} = \cA \spanning \{ a_w : w \in \I\}$ is the $\cH_{q^2}$-module generated by $\I$.
 \end{itemize}
  Recall 
  also the definitions of the special bases $\( C_w \)_{w \in W}\subset \cH_{q^2}$ and  $\(A_w \)_{w \in \I}\subset \cM_{q^2}$, 
 and the polynomials $\( P_{y,w} \)_{y,w \in W}$ and $\( \Psig_{y,w} \)_{y,w \in \I}$ in $ \ZZ[q]$.

\subsection{Bruhat order on twisted involutions}

The set of twisted involutions $\I$ is partially ordered by the Bruhat order $\leq $ on $W$, and this ordering controls many important features of the basis $\(A_w \)_{w \in \I}\subset \cM_{q^2}$ and the polynomials $\( \Psig_{y,w} \)_{y,w \in \I}$. The subposet $(\I,\leq)$ has a more direct characterization and a number of interesting properties, which are meticulously detailed in Hultman's papers \cite{H1,H2,H3}.  Hultman's work extends to arbitrary Coxeter systems many earlier observations of Richardson and Springer \cite{R,RS,S} concerning 
$\I$ when $W$ is finite.
Here we  review some of this material which will be of use later on, particularly in Section \ref{technical-sect}.

Recall from Section \ref{twisted-intro} that we define
\be\label{act-def} s\act  w \omdef = \begin{cases} sw&\text{if }sw=ws^* \\ sws^*&\text{if }sw\neq ws^* \end{cases}
\qquad\text{for $s \in S$ and $w \in \I$.}
\ee
In \cite{LV2}, Lusztig  uses the notation $s \bullet w$ instead of $s \act w$; we prefer the symbol $\act$ to emphasize that $s \in S$ acts to ``twist'' $w \in \I$. Although this notation does not extend to an action of $W$ of $\I$, it does lead to the following definition, adapted from \cite{H1,H2,H3}:

%
\begin{definition}
A sequence 
$ (s_1,s_2,\dots, s_k)$
 with $s_i \in S$ is an \emph{$\I$-expression} for a twisted involution $w \in \I$ if 
$ w = s_1 \act (s_2 \act (\cdots \act (s_k \act 1)\cdots )).$
 An {$\I$-expression} for $w$ is \emph{reduced} if its length $k$ is minimal. We consider the empty sequence $()$ to be a reduced $\I$-expression for  $w=1$.
 \end{definition} 
 What we refer to as $\I$-expressions are the left-handed versions of what Hultman terms  ``$\uS$-expressions'' in \cite{H1,H2,H3}. (In consequence, all of our statements here  are in fact the left-handed versions of Hultman's.) It follows  by induction on  $\ell(w)$ that every $w \in \I$ has a reduced $\I$-expression, and so the next statement (given as \cite[Proposition 2.5]{H3}) is well-defined:
 
 \begin{proposition}[Hultman \cite{H3}]\label{ell-star}
 Choose a reduced $\I$-expression $(s_1,s_2,\dots,s_k)$ for $w \in \I$ and define $w_0=w$ and $w_i = s_i \act w_{i-1}$ for $1\leq i\leq k$. Then 
the number of indices $i \in \{1,2,\dots,k\}$ with 
$s_i w_i =w_is_i^*$
depends only on $w$ and not on the choice of $\I$-expression. \end{proposition}
 
 Define $\ell^* : \I \to \NN$ by setting $\ell^*(w)$ equal to the number  defined in the preceding proposition. (In particular, $\ell^*(1) =0$ and $\ell(s) = 1$ for any $s \in S \cap \I$.)
 The function $\ell^*$ coincides with the map $\phi$ which  Lusztig defines in \cite[Proposition 4.5]{LV2}. This map measures the difference in size between the (ordinary) reduced expressions and reduced $\I$-expressions for a twisted involution, in the  sense  of the following result, which appears as \cite[Theorem 4.8]{H1}.

\begin{thmdef}[Hultman \cite{H1}]  \label{rho-def} Let $\rho : \I \to \NN$ be the map which assigns to $w \in \I$ the common length of any of its reduced $\I$-expressions.  Then the poset $(\I,\leq)$ is graded with rank function $\rho$, and 
 \[\rho = \tfrac{1}{2}\(\ell+\ell^*\).\]
In particular, if $w \in \I$  and $s \in S$ then $\rho (s\act w) = \rho(w) - 1$ if and only if $\ell(sw) = \ell(w)-1$.
\end{thmdef}

We conclude  by stating the  ``subword property'' for the Bruhat order on $\I$, which appears  for arbitrary Coxeter systems as \cite[Theorem 2.8]{H3}.

\begin{theorem}[Hultman \cite{H3}]\label{bruhat-thm} 
If $y,w \in \I$ are twisted involutions, then $y \leq w$ if and only if whenever 
  $(s_1,s_2,\dots,s_k)$ is a reduced $\I$-expression for $w$,
 there exist indices $1 \leq i_1 <i_2 < \dots < i_m \leq k$ such that $(s_{i_1},s_{i_2}\dots,s_{i_m})$ is a reduced $\I$-expression for $y$.
 \end{theorem}

\def\msig{m^\sigma}

\subsection{Multiplication formulas and a recurrence for $\Psig_{y,w}$}\label{algo-sect}

While Theorem-Definition \ref{twisted-thmdef} establishes the existence of the distinguished basis $\( A_w\)_{w \in \I}$ for the $\cH_{q^2}$-module $\cM_{q^2}$, it gives no immediate indication of how $\cH_{q^2}$ acts on this basis, or of how one can compute the polynomials $\(\Psig_{y,w}\)_{y,w \in \I}$.
In this section we summarize the main results of Lusztig \cite{LV2} addressing these problems. 

\begin{notation}
Remember that $q=v^2$. To refer to the coefficients of $\Psig_{y,w}\in \ZZ[q]$ of highest  possible order,
given $y, w \in \I$, we let 
\[
\ba
 \mu^\sigma(y,w) &\omdef= \text{the coefficient of  $v^{\ell(w)-\ell(y)-1}$ in $P^\sigma_{y,w}$}
 \\
  \nu^\sigma(y,w) &\omdef= \text{the coefficient of $v^{\ell(w)-\ell(y)-2}$ in $P^\sigma_{y,w}$.}
  \ea
\]
%
%
%
%
In turn, for each 
 $s \in S$  define another integer $\mu^\sigma(y,w;s)$ by the following more complicated formula:
\[ \mu^\sigma(y,w;s) \omdef= 
\nu^\sigma(y,w)  + \delta_{sy,ys^*} \mu^\sigma(sy,w) -  \delta_{sw,ws^*} \mu^\sigma(y,sw) 
 - \sum_{{x \in \I;\hs sx<x }} \mu^\sigma(y,x) \mu^\sigma(x,w).
 \]
As usual,  the Kronecker delta here means $\delta_{a,b} =1$ if $a=b$ and $\delta_{a,b} = 0$ otherwise. 
\end{notation}

Note since $\Psig_{y,w}$ is a polynomial in $q=v^2$ that $\mu^\sigma(y,w)$ (respectively, $\nu^\sigma(y,w)$) is nonzero only if $y\leq w$ and $\ell(w) - \ell(y)$ is odd (respectively, even).
The numbers $\mu^\sigma(y,w;s)$ have an analogous property, which requires a short argument to prove.
Here and elsewhere, for any $w \in W$ we write 
\be\label{des-def}\Des(w) \omdef = \{ s \in S : \ell(sw) < \ell(w)\}\qquand
\mathrm{Des}_R(w) \omdef = \{ s \in S : \ell(ws) < \ell(w)\}\ee
 for the corresponding left and right descent sets.

\begin{proposition} Let $y,w \in \I$ and $s \in \Des(y) \setminus \Des(w)$. Then the integer $\mu^\sigma(y,w;s)$ is nonzero only if $\ell(w) - \ell(y)$ is even and $y <  s\act w$.
\end{proposition}


 \begin{proof} All terms in the definition of $\mu^\sigma(y,w;s)$ are zero if $\ell(w) - \ell(y)$ is odd. 
Assume $ y \not < s \act w$. Then 
$y \not < w$ 
automatically so $\mu^\sigma(y,w;s) = \delta_{sy,ys^*} \mu^\sigma(sy,w)$.
This is zero unless $sy=ys^*$, but if $sy=ys^*$ then    $sy = s\act y \not < w$, as  $s\act y < w$ would imply the contradiction $y < s \act w$ by Theorem \ref{bruhat-thm}.
(In detail, if $s\act y < w$ then adding $s$ to the beginning of any reduced $\I$-expression for $s \act y$ or $w$  forms a reduced $\I$-expression for $y$ or $s\act w$, respectively.)
Thus  $\mu^\sigma(s\act y,w) = 0$.
 \end{proof}
 

Finally, define $\msig(y \xrightarrow{s}w) \in \cA$ for $y,w \in \I$ and $s \in S$ as the Laurent polynomial 
\be\label{msig}
\msig(y \xrightarrow{s}w) 
= \begin{cases} \mu^\sigma(y,w)(v+v^{-1}) &\text{if $\ell(w) - \ell(y)$ is odd} \\ \mu^\sigma(y,w;s)&\text{if $\ell(w) - \ell(y)$ is even.}\end{cases}
\ee
Lusztig proves the following result, which explains our notation, as \cite[Theorem 6.3]{LV2}.

 \begin{theorem}[Lusztig \cite{LV2}] \label{mult-thm}
Let $w \in \I$ and $s \in S$. 
Then $C_s = q^{-1}(T_s + 1)$ and 
\[C_s A_w = \begin{cases} 
\(q+q^{-1}\) A_w&\text{if $s \in \Des(w)$} \\[-10pt]\\
 \(v+v^{-1}\) A_{s w} +\sum_{{y \in \I ;\hs  sy<y<s w}}  \msig(y \xrightarrow{s}w) A_y&\text{if $s \notin\Des(w)$ and $sw=ws^*$} \\[-10pt]\\
 A_{s w s^*} +\sum_{{y \in \I; \hs sy<y<s ws^*}}  \msig(y \xrightarrow{s}w) A_y&\text{if $s \notin \Des(w)$ and $sw\neq ws^*$}.\end{cases}\]
%
%
%
\end{theorem}


We may equivalently rewrite this theorem as a recurrence for the polynomials $\Psig_{y,w}$. This provides the following ``twisted'' analog of one of the standard recurrences (see, e.g., \cite[Theorem 5.1.7]{CCG})  for the ordinary Kazhdan-Lusztig polynomials $P_{y,w}$.

\begin{corollary}\label{main-recurrence}
Let $y,w \in \I$ with $y \leq w$  and  $s \in \Des(w)$.  

\begin{enumerate}

\item[(a)] $\Psig_{y,w} = \Psig_{s\act y,w}$.

\item[(b)] If $ s \in \Des(y)$ and   $w'=s\act w$ and  $c = \delta_{sw,ws^*}$ and $d = \delta_{sy,ys^*}$, then
{
\[\label{partb}  (q+1)^c \Psig_{y,w} = 
(q+1)^d P_{s\act y,w'}^\sigma  +  q(q-d)P_{ y,w'}^\sigma 
- 
\sum_{\substack{z \in \I;\hs sz<z \\ y \leq z < w}} v^{\ell( w)-\ell(z)+c}\cdot \msig(z \xrightarrow{s} w') \cdot P_{ y,z}^\sigma.
\]
}
\end{enumerate}
\end{corollary}

\begin{proof}
The corollary results from comparing coefficients of $a_y$ on both sides of the  equation in Theorem \ref{mult-thm}. 
Rewriting the right hand side in the standard basis $\(a_w\)_{ w\in \I}$ is straightforward from the definitions in Section \ref{twisted-intro}, while rewriting the left hand side can be done using the identities $C_s = q^{-1}(T_s +1)$ and $A_w = v^{-\ell(w)} \sum_{y \in \I} \Psig_{y,w} a_y$ with the multiplication rule \eqref{module-def}.
\end{proof}


 Translating this corollary  into an algorithm for computing the polynomials $\Psig_{y,w}$ involves a little subtlety, because  terms on the right hand side of the recurrence in part (b) can  depend on $\Psig_{y,w}$. The companion paper \cite{EM2} discuss these issues in detail (and provides pseudocode for the resulting algorithm).

 One can establish several notable properties of the polynomials $\Psig_{y,w}$   using Corollary \ref{main-recurrence} and induction on $\ell(w)$. For example, we have these facts mentioned by Lusztig in \cite{LV2}:

\begin{corollary}[Lusztig \cite{LV2}]\label{lv-cor}
Let $y,w \in \I$ with $y\leq w$.
\begin{enumerate}
\item[(a)] $\Psig_{y,w}$ has constant coefficient 1.
\item[(b)] $\Psig_{y,w} = \Psig_{y^{-1},w^{-1}} = \Psig_{y^{*},w^{*}}$.
\end{enumerate}
\end{corollary}

Part (a), which Lusztig states explicitly as \cite[Proposition 4.10]{LV2}, mirrors the fact that $P_{y,w}$ has constant coefficient 1 whenever $y\leq w$ in $W$. Thus the polynomials $P^+_{y,w}$ and $P^-_{y,w}$ (see \eqref{pm-def})  have constant coefficients 1 and 0, respectively, whenever $y \leq w$ in $\I$.
The following lemma states  two other properties of the ordinary Kazhdan-Lusztig polynomials which we will have cause to refer to. These results all appear, for example, in  \cite[Chapter 7]{Hu}.

 \begin{lemma}\label{kl-cor}
Let $y,w \in W$. 
\begin{enumerate}
\item[(a)] $P_{y,w} = P_{sy,w}$ if $s \in \Des(w)$   and $P_{y,w} = P_{ys,w}$ if $s \in \mathrm{Des}_R(w)$. 
\item[(b)] $P_{y,w} = P_{y^{-1},w^{-1}} = P_{y^*,w^*}$.
\end{enumerate}
\end{lemma}

\subsection{Parity statements}\label{parity-sect}

Conjectures \descref{A$'$}, \descref{B$'$} and \descref{C$'$} are statements concerning whether the Laurent polynomials
\[ P^\pm_{y,w} = \tfrac{1}{2} \( P_{y,w} \pm \Psig_{y,w}\)\qquand h^\pm_{x,y;z} = \tfrac{1}{2} \( \wt h_{x,y;z} \pm h^\sigma_{x,y;z}\)\]
defined by equations \eqref{pm-def}, \eqref{h-def}, and \eqref{hpm-def}
for $x \in W$ and $y,z,w \in \I$ have nonnegative integer coefficients.  It is not clear \emph{a priori} that these polynomials even have integer coefficients, 
and we spend this last preliminary section clarifying this property.

Here we write $f\equiv g \modu 2)$ for two Laurent polynomials $f,g \in \cA$ if $f-g$ has only even integer coefficients; i.e., if $f-g=2h$ for some $h \in \cA$.
Lusztig  proves the following result, showing  that $P^\pm_{y,w} \in \ZZ[q]$, as \cite[Theorem 9.10]{LV2}.

\begin{proposition}[Lusztig \cite{LV2}]\label{p-parity} For all $y,w \in \I$ we have $P_{y,w} \equiv \Psig_{y,w} \modu 2)$.
\end{proposition}


The next proposition shows likewise that $h^\pm_{x,y;z} \in \ZZ[v,v^{-1}]$. In the case that $(W,S)$ is a Weyl group and $*$ is trivial, this property was mentioned without proof in \cite[Section 5]{LV}.

\begin{proposition}\label{h-parity} For all
 $x \in W$ and $y,z \in \I$
we have $\wt h_{x,y,z} \equiv  h^\sigma_{x,y,z} \modu 2)$.
\end{proposition}

\begin{proof}
In what follows it is helpful to recall  that we denote the bases of $\cH_q$ using the lower case symbols $t_w$, $c_w$  
and the bases of $\cH_{q^2}$ using the upper case symbols $T_w$, $C_w$.  
 Let $w^\dag = {w^*}^{-1}$ for $w \in W$ and 
define $h \mapsto h^\dag$ as the unique $\cA$-algebra anti-automorphism of $\cH_q$ such that $(t_w)^\dag = t_{w^\dag}$. Also write $\proj : \cH_q \to \cM_{q^2}$ for the $\cA$-linear map with $t_w \mapsto a_w$ for $w \in \I$ and $t_w\mapsto 0$ for $w \in W \setminus \I$.

We write $m \equiv m' \modu 2)$ for  $m,m' \in \cM_{q^2}$ if  $m-m' = 2m''$ for some $m'' \in \cM_{q^2}$. With respect to this notation,
Lusztig  \cite[9.4(a)]{LV2} proves that
\be\label{lus-lem} \proj\(t_x t_y (t_{x})^\dag\)  \equiv T_{x} a_y  \modu 2)\qquad
\text{for all $x \in W$ and $y \in \I$.}\ee
The current proposition derives from this fact 
 in the following way. Let $x \in W$ and $y \in \I$ and 
note that $(c_w)^\dag = c_{w^\dag}$ 
 by Lemma \ref{kl-cor}.  
The anti-automorphism $\dag$ consequently preserves $c_x c_y c_{x^\dag}$, so we must have $\wt h_{x,y,z} = \wt h_{x,y,z^\dag}$ for all $z \in W$ and it follows that we can write 
$c_x c_y c_{x^\dag} = (a+a^\dag) + \sum_{z \in \I} \wt h_{x,y,z} c_z$
for an element $a \in \cH_q$.   
Since $\proj(a+a^\dag) \equiv 0 \modu 2)$ and since $\proj(c_z) \equiv A_z \modu 2)$ for  $z \in \I$ by Proposition \ref{p-parity}, we deduce that
\be\label{2.2a} \proj(c_x c_y c_{x^\dag}) \equiv  \sum_{z \in \I} \wt h_{x,y,z} A_z \modu 2).\ee
On the other hand,  by definition 
$ c_x c_y c_{x^\dag} =   v^{-2\ell(x)-\ell(y)}\sum_{x',z,x'' \in W} P_{x',x}  P_{z,y} P_{x'',x} \cdot t_{x'}t_{z} (t_{x''})^\dag 
.$
Since $P_{z,y} = P_{z^\dag,y}$ for all $z \in W$ as $y = y^\dag$, the anti-automorphism $\dag$ acts on the latter sum by exchanging the summands indexed by $(x',z,x'')$ and $(x'',z^\dag,x')$. It follows by
dividing the sum $\sum_{x',z,x'' \in W} $ into two parts, consisting of the summands  fixed and unfixed by $\dag$, that 
we can write
\[ c_x c_y c_{x^\dag} =  (b+b^\dag) + v^{-2\ell(x)-\ell(y)}\sum_{x' \in W} \sum_{z \in \I} (P_{x',x})^2 P_{z,y}  \cdot t_{x'} t_{z} (t_{x'})^\dag 
.
\]
for another element $b \in \cH_q$.
Since $\proj(b+b^\dag) \equiv 0 \modu 2)$ and $(P_{x',x})^2 \equiv P_{x',x}(q^2) \modu 2)$ and  $P_{z,y} \equiv \Psig_{z,y} \modu 2)$ for $y,z \in \I$,
applying \eqref{lus-lem} to the preceding equation shows that 
\be\label{2.2b}\proj\(c_x c_y c_{x^\dag}\) \equiv C_{x} A_y \modu 2)\qquad\text{for all $x \in W$ and $y \in \I$}.\ee
The proposition now follows immediately by combining \eqref{2.2a} and \eqref{2.2b}.
\end{proof}

\section{Positivity results for universal Coxeter systems}\label{3sect}

In this section we prove our main results. Thus, for the duration  we let $(W,S)$ denote a fixed universal Coxeter system and  we let $*$ denote a fixed $S$-preserving involution of $W$. 
%
It is helpful to recall that the involution $*$ of $W$   corresponds to an arbitrary choice of a permutation with order $\leq 2$ of the set $S$.
The twisted involutions $w \in \I = \{ x \in W : x^{-1} = x^*\} $  each take one of two possible forms:
\begin{itemize}
\item If $\ell(w)$ is even then $w=x^{*}x^{-1}$ for some $x \in W$.
\item If $\ell(w)$ is odd then $w=x^{*}sx^{-1}$ for some $x \in W $ and $s \in S$ with $s = s^*$. 
\end{itemize}
 The following observation enumerates a few other special properties of universal Coxeter systems which make them
 tractable test cases for  general questions and conjectures. 
Recall here the definition of  $s\act w$ from \eqref{act-def}.

\begin{observation} Assume $(W,S)$ is a universal Coxeter system.
\ben
\item[(a)] Each $w \in W$ has a unique reduced expression.
\item[(b)] Each $w \in \I$ has a unique reduced $\I$-expression.
\item[(c)] If $w \in W \setminus\{1\}$ then $|\Des(w)| = |\mathrm{Des}_R(w)|= 1$.
\item[(d)] The map $S \times \I \to \I$ given by 
$ (s,w) \mapsto s\act  w$
%
  extends to a group action of $W$ on $\I$.
\een 
\end{observation}

\begin{notation}
In light of part (d), it is well-defined to 
set
$ x \act w \omdef= s_1 \act (s_2 \act (\cdots \act (s_n \act w)\cdots ))$
where $x \in W$ and $w \in \I$ and $s_i \in S$ are such that  $x=s_1s_2\cdots s_n$.  
\end{notation}

Before proceeding, we note as
a   consequence of our observation that in the special case that  $*$ has no fixed points in $S$, one can view the ordinary Kazhdan-Lusztig theory 
of a universal Coxeter system as a special case of its twisted theory, in the following way 

\begin{observation} Suppose $(W,S)$ is a universal Coxeter system and $s\neq s^*$ for all $ s \in S$.
Then the unique $\cA$-linear map $\cH_{q^2}\to \cM_{q^2}$ with $T_{w^{*}} \mapsto a_{w^{*} w^{-1}}$ for  $w\in W$ defines an isomorphism of left $\cH_{q^2}$-modules which commutes with the bar operators of $\cH_{q^2}$ and $\cM_{q^2}$, 
and
%
consequently
\[\Psig_{(y^{*}y^{-1}), (w^{*}w^{-1})}= P_{y,w}(q^2)
\quand 
h^\sigma_{x,\hs (y^{*}y^{-1});\hs (z^{*}z^{-1})} = h_{x,y^{*};z^{*}}(v^2)
\qquad\text{for all $w,x,y,z \in W$.}
\]
\end{observation}

\begin{proof}
If 
$s\neq s^*$ for all $ s \in S$
then
every $w \in \I$ has even length and the map $w^{*} \mapsto w^{*}w^{-1}$ defines a poset isomorphism $(W,\leq) \xrightarrow{\sim} (\I,\leq)$. From this, the proof of the proposition is a straightforward  exercise using  
Theorem-Definitions \ref{kl-thmdef}, \ref{lv-thmdef}, and \ref{twisted-thmdef}. 
%
\end{proof}


%

\subsection{Kazhdan-Lusztig polynomials}\label{3a-sect}

Dyer derived  a formula for the Kazhdan-Lusztig polynomials of a universal Coxeter system \cite[Theorem 3.8]{Dyer} which shows their coefficients to be nonnegative. We review the key parts of this result here.
To begin, 
we note the following lemma which is a special case of \cite[Lemma 3.5]{Dyer}. 
 
 \begin{lemma}[Dyer \cite{Dyer}] \label{dyer-lem} 
  Assume $(W,S)$ is a universal Coxeter system.  Suppose $y,w \in W$ and $r,s \in S$ such that  $rsw<sw<w$ and $sy>y$.
Then  
\[P_{ y, w} =
 P_{ y,s w} + qP_{s  y,s w}  
-
\delta \cdot qP_{  y,rsw}\qquad\text{where }\delta = |\{ s\} \cap \Des(rsw)|.\]

\end{lemma}


In the  sequel we adopt the following notation.
 Given $y,z,w \in W$ with $y \leq z$,  define 
\be\label{yzw} P_{y,z;w} \omdef= P_{y,w} - P_{z,w}.\ee
%
%
%
 We expand upon the previous lemma with the following statement.
 
 \begin{proposition}\label{dyer-cor2}
  Assume $(W,S)$ is universal.
 Let $y,z \in W$ with $y \leq z$ and suppose 
 \begin{itemize}
 \item $k$ is a positive integer;
 \item $r,s \in S$ such that $r\neq s$ and $s \notin \Des(y)$ and $s\notin \Des(z)$;
 \item $u \in W$ such that $\{ r,s \} \cap \Des(u) = \varnothing$.
 \end{itemize}
If $a,w \in W$ are defined by
 \[ w = \underbrace{srsrs\cdots}_{k+1\text{ factors}}   u
 \qquand
 a = \underbrace{\cdots srsrs}_{k\text{ factors}}
 \] 
then 
$P_{y,z;w} = P_{y,z;sw} + q^k P_{ay,az;aw}$.

 \end{proposition} 
 
\begin{remark} Applying the identity $P_{y,z;w} = P_{{y^*}^{-1},{z^*}^{-1};{w^*}^{-1}}$ from Lemma \ref{kl-cor} affords a right-handed version of this proposition, which will be of use in Section \ref{technical-sect} below.
\end{remark}
 
 \begin{proof} Note  that  $y \leq z$ implies  $ay \leq az$, since (as $sy>y$ and $sz>z$) the unique reduced expression for $ay$ (respectively, $az$) is formed    by concatenating $\cdots srsrs$ to the unique reduced expression for $y$ (respectively, $z$). 
  %
To prove the lemma, we proceed by induction on $k$.  
 If $k=1$ then the lemma reduces to Lemma \ref{dyer-lem}. If $k>1$, then since $P_{sy,sz;rsw} = P_{y,z;rsw}$ by Lemma \ref{kl-cor}, Lemma \ref{dyer-lem}  asserts that
$ P_{y,z;w} = P_{y,z;sw} + q(P_{sy,sz;sw} - P_{sy,sz;rsw}).$
 By induction we may assume that
$P_{sy,sz;sw} = P_{sy,sz;rsw} + q^{k-1} P_{ay,az;aw}$; substituting this identity into the preceding equation  gives the desired recurrence.
  \end{proof}

From the last lemma we have an easy proof of Conjecture \descref{B} (and so also of Conjecture \descref{A}) for universal Coxeter systems. This result can also be deduced from  \cite[Theorem 3.8]{Dyer}. 

 \begin{theorem}[Dyer \cite{Dyer}]\label{dyer-thm1}
If $(W,S)$ is a universal Coxeter system, then the polynomial $P_{y,w} - P_{z,w}$ has nonnegative integer coefficients for all $y,z,w \in W$ with $y \leq z$ in the Bruhat order. 
In particular, we have $P_{y,w} \in \NN[q]$ for each $y,w \in W$. 
 \end{theorem} 
 
 \begin{proof} The proof is by induction on $\ell(w)$. If $\ell(w) \leq 1$ then $P_{y,z;w} \in \{0,1\} \subset \NN[q]$ by Lemma \ref{kl-cor}.
 Assume $\ell(w) \geq 2$ so that there exists $s \in \Des(w)$. Let $(y',z') $ be the unique pair in the set $ \{ (y,z), (sy,z), (y,sz), (sy,sz) \}$ which has $s \notin \Des(y')$ and $s \notin \Des(z')$. It is straightforward to check that $y' \leq z'$, and by Lemma \ref{kl-cor} we have $P_{y,z;w} = P_{y',z';w}$.
%
Proposition \ref{dyer-cor2} applied $P_{y',z';w}$ shows that $P_{y,z;w} \in \NN[q]$ by induction, and it follows that $P_{y,w} \in \NN[q]$ since $P_{w,w} = 1$. \end{proof}

 \subsection{Twisted Kazhdan-Lusztig polynomials}\label{3b-sect}

%
%

%
%

Here we initiate the proof of Conjecture \descref{B$'$} for universal Coxeter systems, to be completed in the next section. As above,  $(W,S)$ is a fixed Universal Coxeter system with a fixed $S$-preserving involution $*$.

Recall the definition of the Laurent polynomial $\msig(y\xrightarrow{s}w) \in \cA$ from \eqref{msig}. 

\begin{lemma} \label{msig-lem} Assume $(W,S)$ is a universal Coxeter system. If $y,w \in \I$ and $r,s \in S$ such that $y\leq w$ and
$\Des(y) = \{s\} \neq \{r\} = \Des(w)$, then \[\msig(y\xrightarrow{s}w) = \begin{cases} 1&\text{if $y =rwr^*$ or if $(y,w) = (s,r)$} \\0&\text{otherwise}.\end{cases}\]

%
%
%
%
%
\end{lemma}

\begin{proof}
First note that since $W$ is a universal Coxeter group and $y\notin \{1,r\}$ we have  $r \act y = ryr^*$ and $\ell(r \act y) = \ell(y) + 2$. In addition,  from Corollary \ref{main-recurrence} 
we have $\Psig_{y,w} = \Psig_{r\act y ,w}$.

We claim that  $\mu^\sigma(y,w) = 0$. To prove this, note that if $r \act y = w$ then   $\ell(w)-\ell(y)$ is even, and if $r \act y \not \leq w$ then $y \not < w$, so in either case $\mu^\sigma(y,w) = 0$. On other other hand, if  $r \act y < w$ then the degree of $\Psig_{ y,w}  = \Psig_{r\act y ,w}$ as a polynomial in $q$ is at most $\frac{\ell(w) - \ell(r\act y) - 1}{2} $ which is strictly less than $\frac{\ell(w)-\ell(y) - 1}{2}$, so again $\mu^\sigma(y,w) = 0$.

 It thus suffices to show that $ \mu^\sigma(y,w;s)= 1$  if $y =rwr^*$ or if $(y,w) = (s,r)$ and $ \mu^\sigma(y,w;s)= 0$ otherwise. To this end,  observe that if we apply our first claim to the definition of $\mu^\sigma(y,w;s)$, and also note that $sw\neq ws^*$ since $w \notin \{1,s\}$, we obtain 
 \[ \mu^\sigma(y,w;s) = \nu^\sigma(y,w) +  \delta_{sy,ys^*} \mu^\sigma(sy,w).\] If $ y = rwr^*$ then  $\Psig_{y,w} = \Psig_{w ,w} = 1 $ so $ \nu^\sigma(y,w) = 1$. Alternatively, if $y < w$ and $ y \neq rwr^*$ then it follows as above that $\Psig_{y,w}= \Psig_{r\act y,w}$ has degree strictly less than $\frac{\ell(w) - \ell(y) -2}{2}$ so $\nu^\sigma(y,w) = 0 $. Hence 
 \be\label{1x} \nu^\sigma(y,w) = \begin{cases} 1 & \text{if }y=rwr^* \\ 0&\text{otherwise}.\end{cases}\ee
 In turn, we have $sy= ys^*$ if and only if $y  =s $ (note that $y \neq 1$ by hypothesis), in which case $\mu^\sigma(sy,w)  = \mu^\sigma(1,w)$.  If $w =r $ then $\mu^\sigma(1,w) = 1$ and if $ w\neq r $ then either $w=rr^*$ (in which case $\ell(w) - \ell(1)$ is even) or $r \act 1 \neq w$ (in which case $\Psig_{1,w} = \Psig_{r \act 1,w}$ has degree strictly less than $\frac{\ell(w)-\ell(1)-1}{2}$) so $\mu^\sigma(1,w) = 0$. Thus 
 \be\label{2x} \delta_{sy,ys^*} \mu^\sigma(sy,w) = \begin{cases} 1 & \text{if $(y,w) = (s,r)$} \\ 0&\text{otherwise}.\end{cases}\ee
 Combining \eqref{1x} and \eqref{2x} gives the desired formula for $\mu^\sigma(y,w;s)$. 
  \end{proof}

We now have the following analogue of Lemma \ref{dyer-lem}.

\begin{lemma}\label{dyer-an} Assume $(W,S)$ is a universal Coxeter system.  Suppose $y,w \in \I$ and $r,s \in S$ such that  $rs\act w<s \act w<w$ and $sy>y$.
Then
\[
P_{ y, w}^\sigma =
P_{ y,s \act  w}^\sigma
+
 q^{2}P_{s\act  y,s \act w}^\sigma   
-
\delta \cdot q^2 P_{s\act  y, rs\act w}^\sigma.
+
\delta' \cdot q(P_{1,s \act w}^\sigma - P_{ s,s\act w}^\sigma ).
\] 
where we define
\[ \delta =\begin{cases} 1 & \text{if }s \in \Des(rs \act w) \\ 0 &\text{otherwise}\end{cases}
\qquand
 \delta' = \begin{cases} 1 & \text{if $y=1$ and $s=s^*$ and $w \neq srs$} \\ 0 &\text{otherwise}.\end{cases}
 \]
\end{lemma}

 \begin{proof} Everything follows by combining Lemma \ref{msig-lem} with Corollary \ref{main-recurrence}. It is straightforward to check that the lemma holds if $y \not < w$, so assume $y < w$.
Let $\delta'' = \delta_{sy,ys^*}$ and  note by hypothesis that $sw \neq ws^*$. 
By Corollary \ref{main-recurrence} we  therefore have
 \be\label{thesum} 
\text{ \small
$ \Psig_{y,w} =  \Psig_{y,s\act w} + q^{2} \Psig_{s\act y, s\act w} 
 + 
 \delta''  q (\Psig_{y,s\act w} - \Psig_{s\act y, s \act w})
  - \sum_{\substack{z \in \I; \hs sz<z \\ y \leq z < w }} v^{\ell(w) - \ell(z)}  \msig(z \xrightarrow{s} s\act w)  \Psig_{y,z}$.}
  \ee
From the preceding lemma we know that $\msig(z \xrightarrow{s} s\act w)=1$  if $z = rs \act w$ or $(z,s\act w) = (s,r)$, and $\msig(z \xrightarrow{s} s\act w)=0$ otherwise.  The sum in \eqref{thesum} includes a summand indexed by $z = rs \act w$ if and only if $\delta = 1$. On the other hand, if $s \act w = r$ then the sum includes a summand indexed by $z = s$ 
if and only if $s=s^*$.
Since $\Psig_{y,s}  =1 $ if $y \in \{1,s\}$ and $\Psig_{y,s} =0$ otherwise, we conclude that 
\[
 P_{ y, w}^\sigma =
P_{ y,s \act  w}^\sigma
+
 q^{2}P_{s\act  y,s \act w}^\sigma   
+
\delta''\cdot  q(\Psig_{y,s\act w} - \Psig_{s\act y, s\act w})
-
\delta \cdot q^2 \Psig_{  y, rs\act w}
-
\delta''' \cdot q
\]
 where we define
 \[ \delta''' = \begin{cases} 1 & \text{if $y=1$ and $s=s^*$ and $w = srs$} \\ 0 &\text{otherwise}.\end{cases}\] 
  Note that if $\delta = 1$ then $\Psig_{y,rs\act w} = \Psig_{s\act y,rs\act w}$ by Corollary \ref{main-recurrence}.
 Thus to finish our proof, it is enough to check that 
 \[ \delta'' (\Psig_{y,s\act w} - \Psig_{s\act y, s\act w})
-
\delta''' 
= \delta'  ( \Psig_{1,s\act w} - \Psig_{s,s\act w})
\]
This is clear if
$y=1$ and $s=s^*$ and $w \neq srs$ since
 then $\delta' = \delta'' = 1$ and $\delta''' = 0$.
 On the other hand, if $y= 1$ and $ s = s^*$ but $w = srs$ then $\delta' = 0$ and $\delta'' = \delta''' = 1$ and
 $\Psig_{y,s\act w} - \Psig_{s\act y, s\act w} = \Psig_{1,r} = 1$, which again gives equality.
 Finally, if $y \neq 1$ or $s \neq s^*$ then $\delta' =\delta'' = \delta''' = 0$ and our equation again holds.
 \end{proof}

%
%
%
%
%
%
%
%
%
%
%
%

\subsection{Four technical propositions}
\label{technical-sect}
 
 To prove Conjectures \descref{A$'$} and \descref{B$'$} for the universal Coxeter system $(W,S)$ we require an analog of Proposition \ref{dyer-cor2}. The requisite statement splits into four somewhat more technical propositions, which we prove here.
The Coxeter system $(W,S)$ is always assumed to be universal in this section (and we stop stating this condition in our results).

 Mirroring the notation $P_{y,z;w}$ from \eqref{yzw}, given $y,z,w \in \I$ with $y \leq z$, we define 
 \be\label{yzwsig} P_{y,z;w}^\sigma \omdef= P_{y,w}^\sigma - P_{z,w}^\sigma.\ee
 Also, given elements $w_1,w_2,\dots ,w_k \in W$ we write $\langle w_1,w_2,\dots,w_k\rangle$ for the subgroup they generate.
 Finally, recall from Theorem-Definition \ref{rho-def} that we denote the rank function on $(\I,\leq)$ by 
\[\rho : \I \to \NN,\]
 so that $\rho(w)$ is the length of any reduced $\I$-expression for $w \in \I$.
  
  At least one half of the following result is well-known, being equivalent to the fact that the Kazhdan-Lusztig polynomials of dihedral Coxeter systems are all constant. 

\begin{proposition}\label{p1} Let $y,z,w \in \I$ with $y\leq z$. If $r,s \in S$ such that
 $w \in \langle r,s \rangle $, 
 then 
\[P^\sigma_{y,z;w} =P_{y,z;w} = \begin{cases} 1&\text{if $y\leq w$ and $z \not \leq w$} \\ 0&\text{otherwise}.\end{cases}\]
\end{proposition}

\begin{proof}
If suffices to show that $\Psig_{y,w} =P_{y,w} = 1$ if $y \leq w$; however, this follows by
a straightforward argument using induction on the length of $w$ and Lemmas \ref{dyer-lem} and \ref{dyer-an}. In particular, the base cases for our induction are given by  Corollary \ref{main-recurrence} and Lemma \ref{kl-cor}, which show that $\Psig_{y,w}=1$ if $y\leq w$ 
and  $\rho(w) \leq 1$, and that  $P_{y,w}=1$ if $y\leq w$ and $\ell(w) \leq 1$. 
\end{proof}

For the duration of this section we adopt the following specific setup: fix $y,z \in \I$ with $y \leq z$ and assume $w \in \I$ has the form
\be w = \underbrace{srsrs\cdots}_{k+1\text{ factors}} \act u\ee
where
\begin{itemize}
\item $k$ is a positive integer;
\item $r,s \in S$ such that $r\neq s$ and $s \notin \Des(y) $ and $s \notin \Des(z)$;
\item $u \in \I$ such that $\{r,s\} \cap \Des(u) = \varnothing$.
\end{itemize}
In addition, define $a \in \langle r,s\rangle \subset W$ as the element
\be\label{a} a = \underbrace{\cdots srsrs}_{k\text{ factors}} \ee
and let  $y',z',w' \in \I$ denote the twisted involutions 
\be
 y' = a\act y
\qquand
z' = a \act z
\qquand
w' = a\act w .\ee
Observe that $\rho(y') = \rho(a) + \rho(y)$ and $\rho(z')  = \rho(a) + \rho(z)$ and $\rho(w') = \rho(u) +1$, and that clearly $y' \leq z'$ in the Bruhat order. In addition, $w' $ is given by either $s \act u$ or $r \act u$, depending on the parity of $k$.
We now have our second proposition.

%
%

\begin{proposition}\label{p2}
 Suppose $w \notin \langle r,s \rangle $ and either $y \neq 1$ or $s\neq s^*$.
 Then 
\ben
\item[(a)]
$\Psig_{y,z;w} = 
\Psig_{y,z;sws^*} + q^{2k} \Psig_{y',z'; w'}$.

\item[(b)]
$
P_{y,z;w} =
P_{y,z; s ws^*} + q^{2k} P_{y',z'; w'} + 2q^k P_{ay,az;aws^*}$.
\een
\end{proposition}

\begin{remark} The best way of making sense of this and the next two propositions is through pictures. The recurrences in  each proposition are conveniently illustrated as trees whose nodes are labelled by the polynomials $\Psig_{y,z;w}$ or $P_{y,z;w}$ and whose edges are labelled by powers of $q$; see
Figures \ref{p2a}, \ref{p2b}, \ref{p3a}, \ref{p3b}, \ref{p4a}, and \ref{p4b}. 
In these diagrams,
the branches at each level indicate one application of Lemma \ref{dyer-an} or Lemma \ref{dyer-lem}; these lemmas add two or three children to a given node while possibly also canceling a node two levels down the tree. This cancelation accounts for the chains of $k$ single-child nodes, which appear as dashed lines.
\end{remark}

\begin{figure}[h]
{\small
\[  
 \xy<0.0cm,0.0cm> \xymatrix@!0@R=1.0cm@C=1.2cm{
*{} &
P^\sigma_
{y,z;w} &
*{}  &
*{}  &
*{}  &
*{}
\\
P^\sigma_
{y,z;sws^*} \ar @{<-} [ur]^{1} &
*{} &
*{\bullet}
 \ar @{<-} [ul]_{q^2} &
*{}  &
*{}  &
*{}
\\
*{} &
*{} &
*{}  &
*{\bullet}  \ar @{<--} [ul]  &
*{}
\\
*{} &
*{} &
*{}  &
*{}  &
P^\sigma_
{y',z';w'}  \ar @{<-} [ul]_{q^2} 
}\endxy
\]}
\caption{Labelled tree illustrating part (a) of Proposition \ref{p2}.}
\label{p2a}
\end{figure}
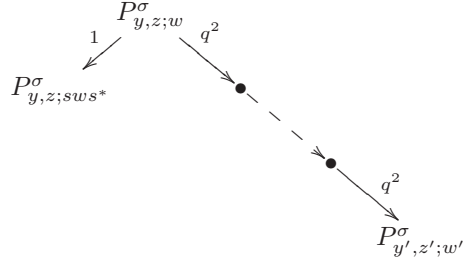

\begin{proof}
First consider Figure \ref{p2a}.
The proof of part (a) is very similar to that of Proposition \ref{dyer-cor2}, but using Lemma \ref{dyer-an} in place of Lemma \ref{dyer-lem}. The argument is entirely analogous because, under our current hypotheses,  whenever we apply Lemma \ref{dyer-an} the second indicator $\delta' $ defined in that result is zero.

Now consider Figure \ref{p2b}. 
To prove part (b), we first apply the right-handed version of Proposition \ref{dyer-cor2} to $P_{y,z;w}$ and then apply the left-handed version of Proposition \ref{dyer-cor2} to the result. In detail, the first application gives
\[P_{y,z;w} = P_{y,z;ws^*} + q^k P_{y{a^*}^{-1},z{a^*}^{-1};w{a^*}^{-1}}\]
while the second  gives
$
P_{y,z;ws^*}= P_{y,z;sws^*} + q^k P_{ay,az;aws^*}
$
and
\[
P_{y{a^*}^{-1},z{a^*}^{-1};w{a^*}^{-1}} = P_{y{a^*}^{-1},z{a^*}^{-1};sw{a^*}^{-1}} + q^k P_{y',z';w'}.
\]
Since $P_{ay,az;aws^*} = P_{y{a^*}^{-1},z{a^*}^{-1};sw{a^*}^{-1}}$ by Lemma \ref{kl-cor} combining the preceding  equations gives the desired recurrence.
%
\end{proof}

\begin{figure}[h]
{\small
\[
 \xy<0.0cm,-0.0cm> \xymatrix@!0@R=1.0cm@C=1.2cm{
*{} &
*{}  &
*{}  &
*{}  &
P_{y,z;w}  &
*{}  &
*{}  &
*{}  &
*{} &
*{}
\\
*{} &
*{}  &
*{}  &
P_{y,z;sw}\ar @{<-} [ur]^{1} &
*{}  &
*{\bullet}  \ar @{<-} [ul]_{q} &
*{}  &
*{}  &
*{} &
*{}
\\
*{} &
*{}  &
*{\bullet}  \ar @{<-} [ur]^{q}  &
*{}   &
P_{y,z;sws^*} \ar @{<-} [ul]_{1}  &
*{}  &
*{\bullet}  \ar @{<--} [ul] &
*{}  &
*{} &
*{}
\\
*{}&
*{\bullet}  \ar @{<--} [ur] &
*{}  &
*{}  &
*{}  &
*{}  &
*{}  &
P_{ay,az;aw}  \ar @{<-} [ul]_{q}&
*{}
\\
P_{y{a^{*}}^{-1}, z{a^{*}}^{-1};sw{a^{*}}^{-1}}  \ar @{<-} [ur]^{q} &
*{} &
*{}  &
*{}  &
*{}  &
*{}  &
*{\bullet}\ar @{<-} [ur]^{q}   &
*{} &
P_{ay,az;aws^*}  \ar @{<-} [ul]_{1}
\\
*{}&
*{} &
*{}  &
&
*{}  &
*{\bullet}\ar @{<--} [ur]   &
*{}  &
*{}  &
*{} &
*{}
\\
*{}&
*{} &
*{}  &
&
P_{y',z';w'} \ar @{<-} [ur]^{q}   &
*{}  &
*{}  &
*{}  &
*{} &
*{}
}\endxy
\]}
\caption{Labelled tree illustrating part (b) of Proposition \ref{p2}. 
} 
\label{p2b}
\end{figure}
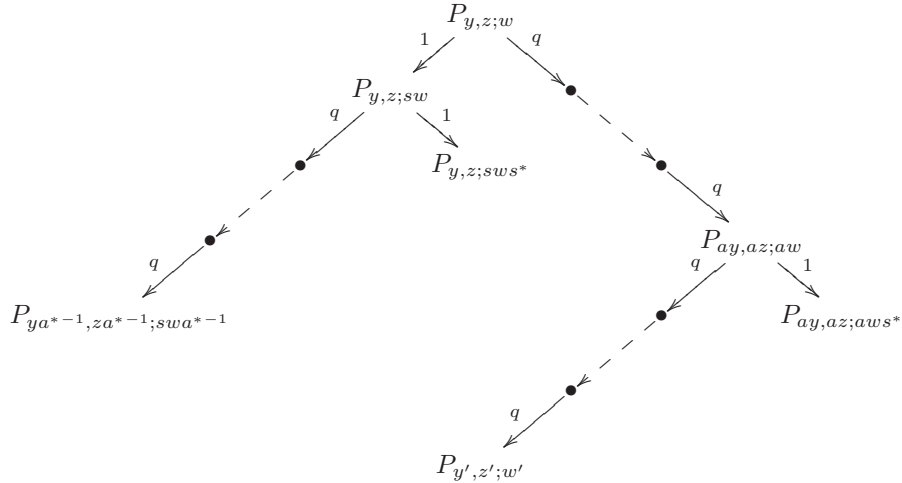

We proceed immediately to our next proposition.

\begin{proposition}\label{p3}
Suppose $w \notin \langle r,s \rangle $ and $y=1\neq z$ and $s=s^*$ and $r=r^*$. Then there are elements 
$u_0,u_1,\dots,u_{k} \in \I $ and $ z_1,z_2,\dots,z_k \in W$ with $u_i \leq u_{i+1}$ and $u_i \leq z_i$ such that
\ben

\item[(a)] 
$P_{y,z;w}^\sigma 
=   P^\sigma_{y,z;sws}+ q^{2k} P^\sigma_{y',z';w'}  + \sum_{0\leq i<k} q^{i+k} P^\sigma_{u_{i},u_{i+1};w'}$.

\item[(b)]
$P_{y,z;w} =
P_{y,z;sws} + q^{2k} P_{y',z';w'} 
+
\sum_{0\leq i<k} q^{i+k} \(P_{u_{i},u_{i+1};w'} + 2 P_{u_{i+1},z_{i+1};w'}\)
.$
\een

\end{proposition}

\begin{proof}
The twisted involutions $u_0,u_1,\dots,u_k \in \I$ are defined as follows:
\begin{itemize} 
\item 
If $k-i$ is even then let $u_i= (\cdots srsrs) \act 1$, where $(\cdots srsrs)$ has $i$ factors.

\item 
If $k-i$ is odd then let 
$ u_i =  (\cdots rsrsr) \act 1$, where $(\cdots rsrsr)$ has $i$ factors.
\end{itemize}

\begin{figure}[h]
{\small
\[  
 \xy<0.0cm,0.0cm> \xymatrix@!0@R=1.0cm@C=2.0cm{
*{} &
P^\sigma_
{1,z,w} &
*{}  &
*{}  &
*{}  &
*{}  &
*{}
\\
P^\sigma_
{1,z;sws} \ar @{<-} [ur]^{1} &
P^\sigma_
{1,s;s \act w}  \ar @{<-} [u]^{q}  
 &
*{\bullet}
 \ar @{<-} [ul]_{q^2} &
*{}  &
*{}  &
*{}
\\
0  \ar @{<-} [ur]^{1} &
P^\sigma_
{1,r;rs \act w}  \ar @{<-} [u]^{q} 
&
*{\bullet}
 \ar @{<-} [ul]_{q^2} &
*{\bullet} \ar @{<-} [ul]_{q^2} &
*{}  &
*{}
\\
0  \ar @{<-} [ur]^{1} &
*{\bullet}  \ar @{<-} [u]^q  &
*{\bullet}
 \ar @{<-} [ul]_{q^2}
&
*{\bullet}
 \ar @{<-} [ul]_{q^2}  &
*{\bullet}  \ar @{<-} [ul]_{q^2}  &
*{}
\\
*{} &
*{\bullet}  \ar @{<--} [u]&
*{}  &
*{\bullet}
 \ar @{<--} [ul]
&
*{\bullet}
 \ar @{<--} [ul] &
*{\bullet}  \ar @{<--} [ul]
\\
0  \ar @{<-} [ur]^{1}  &
P^\sigma_{u_0,u_1;w'}  \ar @{<-} [u]^{q}  &
P^\sigma_{u_1,u_2;w'}
 \ar @{<-} [ul]_{q^2} &
\cdots 
&
P^\sigma_{u_{k-2},u_{k-1};w'}
 \ar @{<-} [ul]_{q^2} &
P^\sigma_{u_{k-1},u_{k};w'}
 \ar @{<-} [ul]_{q^2}   &
P^\sigma_{y',z';w'}  \ar @{<-} [ul]_{q^2}
}\endxy
\]}
\caption{Labelled tree illustrating part (a) of Proposition \ref{p3}}
\label{p3a}
\end{figure}
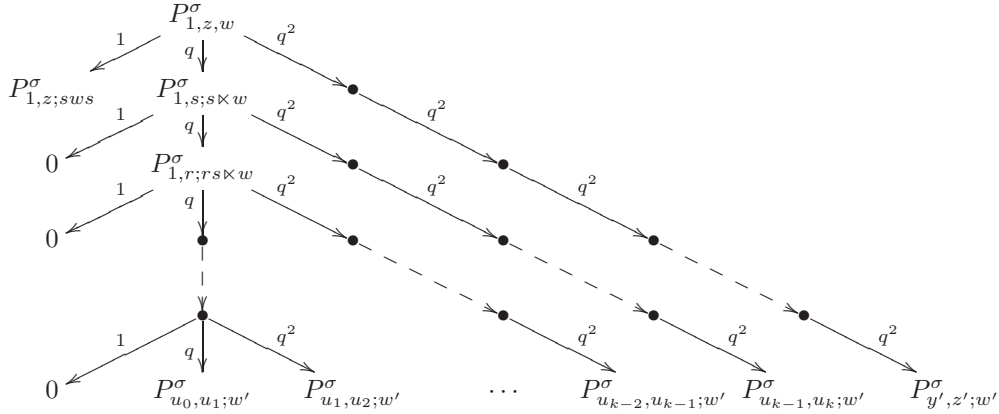

Consider Figure \ref{p3a}.
To prove part (a), we  note that Lemma \ref{dyer-an} implies 
\[ \Psig_{y,z;w} = \Psig_{1,z;s\act w} + q^2 (\Psig_{s,s\act z;s\act w} - \delta \cdot  \Psig_{s,s\act z;rs\act w}) + q \Psig_{1,s;s\act w}\]
where $\delta = 0$ if $k=1$ and $\delta = 1$ otherwise.
If $\delta = 1$ then Proposition \ref{p2} gives 
$\Psig_{s,s\act z;s\act w} = 
\Psig_{s,s\act z;rs\act w} + q^{2(k-1)} \Psig_{y',z';w'}
$; by substituting this into the previous  equation we  get in either case 
\be \label{ind}  \Psig_{y,z;w} = \Psig_{1,z;s \act w} + q^{2k} \Psig_{y',z';w'} + q \Psig_{1,s;s\act w}.\ee
If $k=1$ then this equation coincides with the recurrence in part (a), and if $k>1$ then 
by induction (with the parameters $(k,r,s,y,z,w)$ replaced by $(k-1,s,r,1,s,s\act w)$) we may assume that  
\[ \Psig_{1,s;s\act w} = \Psig_{1,s;rs\act w} + q^{2(k-1)} \Psig_{u_{k-1},u_k;w'} + \sum_{0\leq i < k-1} q^{i+k-1} \Psig_{u_i,u_{i+1};w'}.\]
Since here $\Psig_{1,s;rs\act w} = \Psig_{1,1;rs\act w} = 0$ as $s \in \Des(rs\act w)$, substituting the previous equation into \eqref{ind}
establishes part (a) for all $k$.

Before proving part (b) we must define the elements $z_i \in W$. For this, we first define 
an intermediate sequence $ \wt z_1,\wt z_2, \dots, \wt z_{k+1} \in W$ in the following way.  
Set $\wt z_{k+1} = az$ where $a$ is given by \eqref{a}, and for $i \leq k$ define $\wt z_i$ inductively by these cases:
\begin{itemize}
\item If $k-i$ is even then let $\wt z_i$ be the element with smaller length in the set $\{ \wt z_{i+1}, \wt z_{i+1} r^*\}$.

\item If $k-i$ is odd then let $\wt z_i$ be the element with smaller length in the set $\{ \wt z_{i+1}, \wt z_{i+1} s^*\}$.

\end{itemize}
Note by construction that $\wt z_i r^* > \wt z_i$ if $k-i$ is even and $\wt z_i s^* > \wt z_i$  if $k-i$ is odd. Finally, define $z_1,z_2,\dots,z_k \in W$ as follows:
\begin{itemize}
\item If $k-i$ is even then let $z_i = \wt z_i (rsrsr\cdots)^*$ where $(rsrsr\cdots)$ has $i-1$ factors.

\item If $k-i$ is odd then let $z_i = \wt z_i (srsrs\cdots)^*$ where $(srsrs\cdots)$ has $i-1$ factors.
\end{itemize}
Note that by construction $\ell(z_i) = \ell(\wt z_i) + i-1$. Note also that since we assume  $s=s^*$ and $r=r^*$, the $*$'s in the preceding bullet points are superfluous; however, these will be significant in the proof of the next proposition when we refer  to the definition of $u_{k-1}$, $u_k$, and $z_k$.

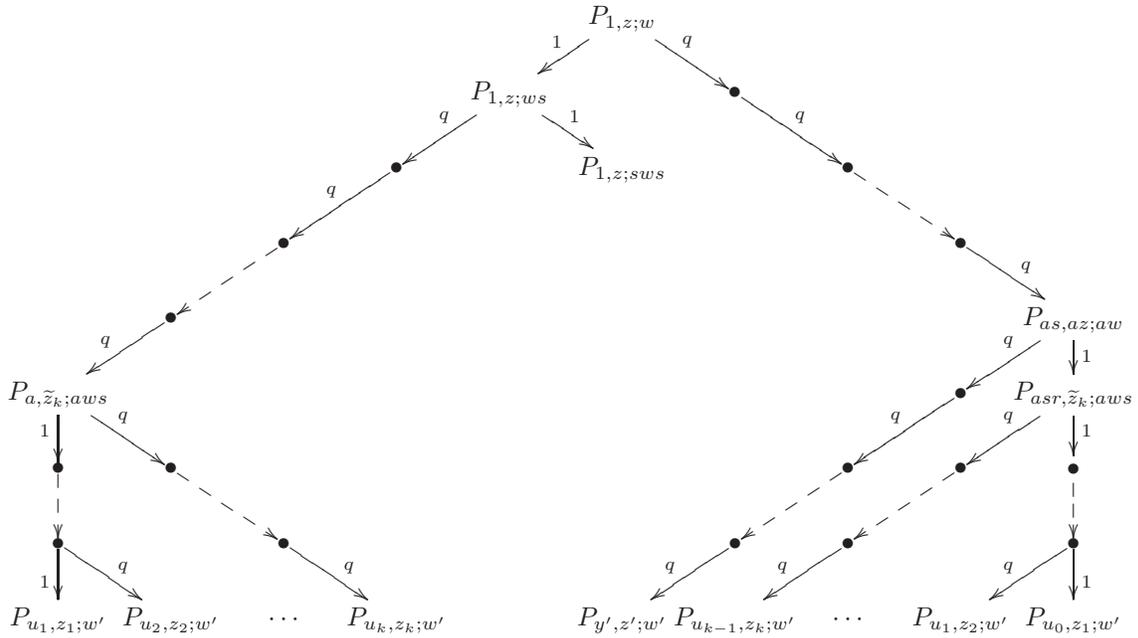
\begin{figure}[h]
{\small
\[  
 \xy<0.0cm,0.0cm> \xymatrix@!0@R=1.0cm@C=1.5cm{
*{} &
*{} &
*{}&
*{}  &
*{}  &
P_{1,z;w}  &
*{}  &
*{}&
*{} &
*{}  
\\
*{}&
*{} &
*{}&
*{}  &
P_{1,z;ws}   \ar @{<-} [ur]^{1} &
*{}  &
*{\bullet}    \ar @{<-} [ul]_{q}&
*{}&
*{} &
*{}  
\\
*{} &
*{} &
*{}&
*{\bullet}   \ar @{<-} [ur]^{q}  &
*{}  &
P_{1,z;sws}   \ar @{<-} [ul]_{1} &
*{}  &
*{\bullet}   \ar @{<-} [ul]_{q}&
*{} &
*{}  
\\
*{} &
*{} &
*{\bullet}   \ar @{<-} [ur]^{q}  &
*{}&
*{}  &
*{} &
*{}  &
*{} &
*{\bullet}   \ar @{<--} [ul]&
*{}  
\\
*{} &
*{\bullet}   \ar @{<--} [ur]  &
*{} &
*{}&
*{}  &
*{} &
*{}  &
*{} &
*{}  &
P_{as,az;aw}   \ar @{<-} [ul]_q&
*{}
\\
P_{a,\wt z_k;aws}   \ar @{<-} [ur]^q  &
*{} &
*{}&
*{}  &
*{} &
*{}  &
*{} &
*{}  &
*{\bullet}   \ar @{<-} [ur]^q&
P_{asr, \wt z_{k};aws}  \ar @{<-} [u]_1
\\
*{\bullet}  \ar @{<-} [u]^1 &
*{\bullet}   \ar @{<-} [ul]_q &
*{}&
*{}  &
*{} &
*{}  &
*{} &
*{\bullet}   \ar @{<-} [ur]^q  &
*{\bullet}   \ar @{<-} [ur]^q&
\bullet \ar @{<-} [u]_1
\\
*{\bullet} \ar @{<--} [u] &
*{} &
*{\bullet}   \ar @{<--} [ul] &
&
*{} &
&
*{\bullet}   \ar @{<--} [ur]  &
*{\bullet}   \ar @{<--} [ur]&
*{}  &
*{\bullet}  \ar @{<--} [u]
\\
P_{u_1, z_1;w'}  \ar @{<-} [u]^1 &
P_{u_2,z_2;w'}   \ar @{<-} [ul]_q &
\cdots&
P_{u_k,z_k ;w'}   \ar @{<-} [ul]_q   &
&
P_{y',z';w'}   \ar @{<-} [ur]^q &
P_{u_{k-1},z_{k};w'}   \ar @{<-} [ur]^q  &
\cdots  &
P_{u_{1},z_{2};w'}   \ar @{<-} [ur]^q&
P_{u_{0},z_{1};w'}  \ar @{<-} [u]_1
}\endxy
\]}
\caption{Labelled tree illustrating part (b) of Proposition \ref{p3}}
\label{p3b}
\end{figure}

Consider Figure \ref{p3b}. To prove part (b), we  note from Proposition \ref{dyer-cor2} that 
\[  P_{y,z;w} = P_{1,z;sw} + q^k P_{a,az;aw}
= P_{1,z;ws} + q^k P_{as,az;aw}.
\]
Here the second equality follows from  properties in Lemma  \ref{kl-cor} (in particular, the fact that $P_{y,w} = P_{ys,w}$ if $ws<w$).
One checks similarly that  applying (the left- and right-handed versions of) Proposition \ref{dyer-cor2} to the terms on the right gives
\be\label{ind2} P_{y,z;w} = P_{1,z;sws} + q^k P_{a,\wt z_k;aws} + q^k P_{asr,\wt z_k;aws} + q^{2k} P_{y',z';w'}.\ee 
From here, it is a straightforward exercise to check the identities \[ P_{a,\wt z_k;aws} = \sum_{i=0}^{k-1} q^{i} P_{u_{i+1},z_{i+1};w'}
\qquand
P_{asr,\wt z_k;aws}  
 = \sum_{i=0}^{k-1} q^{i} P_{u_{i},z_{i+1};w'}
\]
which on substitution afford the desired recurrence  (since $P_{u_{i-1},z_{i};w'} + P_{u_{i},z_{i};w'} = P_{u_{i-1},u_{i};w'} + 2 P_{u_{i},z_{i};w'}$). 
In particular, one obtains these identities by applying the right-handed version Proposition \ref{dyer-cor2} to the
left hand sides, and  then  applying the proposition again 
to the term in the result with coefficient one,  repeating this process until the third index of every polynomial is $w'$.
%
%
%
%
\end{proof}

For this section's final proposition, it is convenient to let $y'',z'',w'' \in W$ denote the elements
\be y'' =a
\qquand
z'' =
\begin{cases} azr^*&\text{if $r^* \in \mathrm{Des}_R(z)$} \\ az&\text{otherwise}\end{cases}
\qquand
w'' = awsr^*.
\ee
We remark that in the notation of the proof of the previous proposition, the element $z'' = \wt z_k$. Thus we also have $z_k = z'' (rsrsr\cdots)^*$ where $(rsrsr\cdots)$ has $k-1$ factors.

\begin{proposition} \label{p4}
 Suppose   $y = 1 \neq z$ and $s= s^*$ and $r\neq r^*$ (so that automatically $w \notin \langle r,s \rangle $). Then, with  $u_{k-1},u_k \in \I$ and $z_k \in W$  defined as in the proof of Proposition \ref{p3}, we have
\ben
\item[(a)] $ P_{y,z;w}^\sigma 
=  P^\sigma_{y,z;s  ws} +q^{2k} P^\sigma_{y',z';w'} +   q^{2k-1} P^\sigma_{u_{k-1},u_k;w'}$.

\item[(b)]
$P_{y,z;w} 
= P_{y,z;s  ws} +  q^{2k} P_{y',z';w'} + q^{2k-1} \(P_{u_{k-1},u_k;w'} + 2 P_{u_k,z_k;w'} \)+ \begin{cases}    2q^k P_{y'',z'';w''} &\text{if }k>1\\ 0&\text{if }k=1.\end{cases}$
\een
\end{proposition}

\begin{figure}[h]
{\small
\[  
 \xy<0.0cm,0.0cm> \xymatrix@!0@R=1.0cm@C=2.0cm{
*{} &
P^\sigma_
{1,z;w} &
*{}  &
*{}  &
*{}  &
*{}  &
*{}
\\
P^\sigma_
{1,z;sws} \ar @{<-} [ur]^{1} &
P^\sigma_
{1,s;s \act w}  \ar @{<-} [u]^{q}  
 &
*{\bullet}
 \ar @{<-} [ul]_{q^2} &
*{}  &
*{}  &
*{}
\\
0  \ar @{<-} [ur]^{1} &
&
*{\bullet}
 \ar @{<-} [ul]_{q^2} &
*{\bullet} \ar @{<-} [ul]_{q^2} &
*{}  &
*{}
\\
*{} &
&
*{}  &
*{\bullet}
 \ar @{<--} [ul] &
*{\bullet}  \ar @{<--} [ul]
\\
 &
 &
 &
&
P^\sigma_{u_{k-1},u_{k};w'}
 \ar @{<-} [ul]_{q^2}   &
P^\sigma_{y',z';w'}  \ar @{<-} [ul]_{q^2}
}\endxy
\]}
\caption{Labelled tree illustrating part (a) of Proposition \ref{p4}}
\label{p4a}
\end{figure}

\begin{proof}
Consider Figure \ref{p4a}.
To prove part (a), we first note that  the argument used to show \eqref{ind} in the previous proposition remains valid here and  
gives
  \[ \Psig_{y,z;w} = \Psig_{1,z;s \act w} + q^{2k} \Psig_{y',z';w'} + q \Psig_{1,s;s\act w}.\]
  If $k=1$ then (using the definitions in the proof of Proposition \ref{p3}) we have  $u_0 = 1$ and $u_1 = s$ and so this equation coincides with the desired recurrence. If $k>1$, then since $r\neq r^*$,  we can apply Proposition \ref{p1} with the parameters $(k,r,s,y,z,w)$ replaced by $(k-1,s,r,1,s,s\act w)$ to obtain 
\[ 
\Psig_{1,s;s\act w} = \Psig_{1,s;rs\act w} + q^{2(k-1)} \Psig_{u_{k-1},u_k;w'}.
\]
Here $u_{k-1} = (\cdots rsrsr)\act 1$ where $(\cdots rsrsr)$ has $k-1$ factors and $u_k = (\cdots srsrs)\act 1$ where $(\cdots srsrs)$ has $k$ factors.
Substituting this identity into our formula for $\Psig_{y,z,w} $  then establishes part (a) for all $k$.

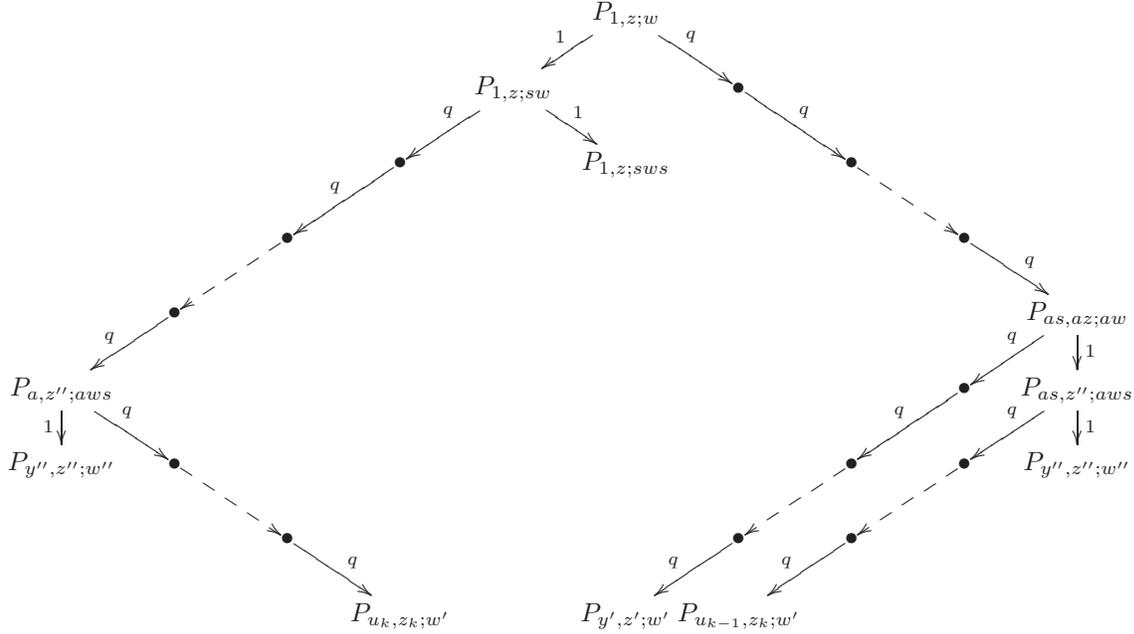
\begin{figure}[h]
{\small
\[  
 \xy<0.0cm,0.0cm> \xymatrix@!0@R=1.0cm@C=1.5cm{
*{} &
*{} &
*{}&
*{}  &
*{}  &
P_{1,z;w}  &
*{}  &
*{}&
*{} &
*{}  
\\
*{}&
*{} &
*{}&
*{}  &
P_{1,z;sw}   \ar @{<-} [ur]^{1} &
*{}  &
*{\bullet}    \ar @{<-} [ul]_{q}&
*{}&
*{} &
*{}  
\\
*{} &
*{} &
*{}&
*{\bullet}   \ar @{<-} [ur]^{q}  &
*{}  &
P_{1,z;sws}   \ar @{<-} [ul]_{1} &
*{}  &
*{\bullet}   \ar @{<-} [ul]_{q}&
*{} &
*{}  
\\
*{} &
*{} &
*{\bullet}   \ar @{<-} [ur]^{q}  &
*{}&
*{}  &
*{} &
*{}  &
*{} &
*{\bullet}   \ar @{<--} [ul]&
*{}  
\\
*{} &
*{\bullet}   \ar @{<--} [ur]  &
*{} &
*{}&
*{}  &
*{} &
*{}  &
*{} &
*{}  &
P_{as,az;aw}   \ar @{<-} [ul]_q&
*{}
\\
P_{a,z'';aws}   \ar @{<-} [ur]^q  &
*{} &
*{}&
*{}  &
*{} &
*{}  &
*{} &
*{}  &
*{\bullet}   \ar @{<-} [ur]^q&
P_{as, z'';aws}  \ar @{<-} [u]_1
\\
P_{y'',z'';w''}  \ar @{<-} [u]^1 &
*{\bullet}   \ar @{<-} [ul]_q &
*{}&
*{}  &
*{} &
*{}  &
*{} &
*{\bullet}   \ar @{<-} [ur]^q  &
*{\bullet}   \ar @{<-} [ur]^q&
P_{y'',z'';w''} \ar @{<-} [u]_1
\\
&
*{} &
*{\bullet}   \ar @{<--} [ul] &
&
*{} &
&
*{\bullet}   \ar @{<--} [ur]  &
*{\bullet}   \ar @{<--} [ur]&
*{}  &
\\
 &
 &
&
P_{u_k,z_k ;w'}   \ar @{<-} [ul]_q   &
&
P_{y',z';w'}   \ar @{<-} [ur]^q &
P_{u_{k-1},z_{k};w'}   \ar @{<-} [ur]^q  &
  &
&
}\endxy
\]}
\caption{Labelled tree illustrating part (b) of Proposition \ref{p4} (when $k>1$)}
\label{p4b}
\end{figure}

To prove part (b),  consider Figure \ref{p4b} and observe that it follows by successive applications of Propositions \ref{dyer-cor2}, exactly as in the proof of Proposition \ref{p3}, that 
\[ P_{y,z;w} = P_{1,z;sws} + q^k P_{a, z'';aws} + q^k P_{as,z'';aws} + q^{2k} P_{y',z';w'}.\] 
Note that the third term on the right $q^k P_{as,z'';aws}$ differs from the analogous equation \eqref{ind2} above; this is because now we have  $asr^*\not< as$ since $r\neq r^*$.
 
Now, if $k=1$ then $u_{k-1} = u_0 = as = 1$ and $u_k = u_1 = a = s$ and $z_k = z_1 = z''$ and $w' = aws$,
 so the preceding formula for $P_{y,z;w}$ coincides with the desired recurrence as $P_{u_{k-1},z_k;w'} + P_{u_k,z_k;w'} = P_{u_{k-1},u_k;w'} + 2P_{u_k,z_k;w'}$. Alternatively, if $k>1$ then the right-handed version of Proposition \ref{dyer-cor2} with the parameters $(k,r,s,y,z,w)$ replaced by $(k-1,s,r,a,z'',aws)$ or $(k-1,s,r,as,z'',aws)$ gives 
 \[  P_{a, z'';aws} = P_{y'',z'';w''} + q^{k-1} P_{u_k,z_k;w''}
 \qquand
  P_{as, z'';aws}  =  P_{as,z'';w''} + q^{k-1} P_{u_{k-1},z_k;w''}. \]
 Since $w''s<w''$ as $k>1$, we have $P_{as,z'';w''} = P_{a,z'';w''} = P_{y'',z'';w''}$, and so substituting these two identities into our previous equation gives the desired recurrence in all cases.
\end{proof}

Our first application of these  results is the following theorem, which shows that the perhaps most natural analogues of Conjectures \descref{A} and \descref{B} for twisted involutions (which are false in general) do hold in the universal case. 

 \begin{theorem}\label{atypical-thm}
If $(W,S)$ is a universal Coxeter system and $* \in \Aut(W)$ is any $S$-preserving involution, then the difference $\Psig_{y,w} - \Psig_{z,w}$ has nonnegative integer coefficients for all $y,z,w \in \I$ with $y \leq z$ in the Bruhat order. In particular,   $\Psig_{y,w} \in \NN[q]$ for each $y,w \in \I$. 
 \end{theorem} 
 
 \begin{proof}
The proof is by induction on $\rho(w)$, and is similar to that of Theorem \ref{dyer-thm1}.
Fix $y,z,w \in \I$ with $y< z$. If $\rho(w) \leq 1$  then the theorem follows from Proposition \ref{p1}. Suppose $\rho(w) \geq 2$, and that $s \in \Des(w)$.
By  Corollary \ref{main-recurrence} we may assume  that $s \notin \Des(y)$ and $s\notin \Des(z)$, in which case one  checks that the triple $(y,z,w)$ satisfies the hypotheses of one of  Propositions \ref{p1}, \ref{p2}, \ref{p3}, or \ref{p4}. These propositions then imply $\Psig_{y,z;w} \in \NN[q]$ by induction.
 \end{proof}


Next, as the main result of this section we  prove that  Conjectures \descref{A$'$} and \descref{B$'$} hold for universal Coxeter systems.

\begin{theorem}\label{typical-thm} If $(W,S)$ is a universal Coxeter system and $* \in \Aut(W)$ is any $S$-preserving involution, then the  polynomials $P^+_{y,w}-P^+_{z,w}$ and $P^-_{y,w}-P^-_{z,w}$ have nonnegative integer coefficients for all $y,z,w \in \I$ with $y \leq w$ in the Bruhat order. 
In particular,   $P^+_{y,w} \in \NN[q]$ and $P^-_{y,w} \in \NN[q]$ for each $y,w \in \I$. 
\end{theorem}

\begin{proof} Recall that the coefficients of $P_{y,z;w} \pm \Psig_{y,z;w}$ are all even by Proposition \ref{p-parity}.
Since $P_{y,z;w}$ and $\Psig_{y,z;w}$ both have positive coefficients  by Theorems \ref{dyer-thm1} and \ref{atypical-thm}, it suffices just to show that $P_{y,z;w} - \Psig_{y,z;w} \in \NN[q]$ for $y,z,w \in \I$ with $y< z$. One can prove this fact by induction on $\rho(w)$ using  the same argument as in the proof of Theorem \ref{atypical-thm}. The same inductive argument works because the differences between parts (a) and (b) in each of our propositions in this section involves only  polynomials $P_{y,z;w} \in \NN[q]$ and  differences $P_{y,z;w} - \Psig_{y,z;w}$.
\end{proof}

\subsection{Structure constants}\label{structure-sect}

In the rest of this paper, we redirect our focus to Conjecture \descref{C$'$}. 
Continue to assume $(W,S)$ is a universal Coxeter system.
This section describes an inductive method of computing the Laurent polynomials $\( h_{x,y;z}\)_{x,y,z \in W}$ and $\( h^\sigma_{x,y;z}\)_{x \in W,\hs  y,z\in \I}$, which we recall from \eqref{h-def} are the structure constants in $\cA = \ZZ[v,v^{-1}]$ satisfying
\[ c_x c_y = \sum_{z \in W} h_{x,y;z} c_z \in \cH_q \qquand C_x A_y = \sum_{z \in \I} h^\sigma_{x,y;z} A_z \in \cM_{q^2}.\]
%
%
We begin by recollecting some relevant  results of Dyer \cite{Dyer} concerning $h_{x,y;z}$ in the universal case. The following  appears as \cite[Definition 3.11]{Dyer}. 

\begin{definition}\label{dyerdef} 
Assume $(W,S)$ is a universal Coxeter system. Let $w \in W$ and $n = \ell(w)$, and  suppose   $s_i \in S$ such that $w =s_1s_2\cdots s_n$.
For each integer $j \in \ZZ$,   define $c(w,j) \in \cH_q$ recursively according to the following cases:

\ben
\item[(a)] If $ 2 \leq j \leq n-1$ (so that $n\geq 3$) and $s_{j-1} = s_{j+1}$, then  set
\[ c(w,j) = c_{w'} + c(w',j-1),\qquad\text{where }w' = s_1 \cdots \widehat{s}_j \widehat{ s}_{j+1} \cdots s_n.\]
Here, we write $\widehat{s}_j$ to indicate that the factor $s_j$ is omitted.

\item[(b)] Otherwise  set $c(w,j) = 0$. 

\een
\end{definition}

The following result of Dyer \cite[Theorem 3.12]{Dyer} gives the decomposition of the product $c_xc_y$ in terms of the Kazhdan-Lusztig basis of $\cH_q$, and shows that the Laurent polynomials $\( h_{x,y;z}\)_{x,y,z \in W}$ have nonnegative coefficients, and are in fact polynomials in $v+v^{-1}$ with nonnegative integer coefficients. (This latter property fails for other Coxeter systems.) 
%

\begin{theorem}[Dyer \cite{Dyer}] \label{c-structure}
Assume $(W,S)$ is  universal. 
Let $x,y \in W$ and $n = \ell(x)$. Then
\[ c_xc_y = \begin{cases} 
 (v  +  v^{-1}) \( c_{xsy} + c(xsy,n) \)
 &
 \text{if
$\mathrm{Des}_R(x) = \Des(y) = \{s\} \neq \varnothing$}
\\
 c_{xy} + c(xy,n) + c(xy,n+1)
&
\text{
otherwise.
}
\end{cases}
\]
%
%

\end{theorem}

\begin{remark} The preceding theorem differs from the corresponding statement in \cite{Dyer} as a result of our notational conventions. In \cite[Theorem 3.12]{Dyer}, Dyer writes ``$C_w$'' to denote the element of $\cH_q$ which in our notation is written
\[ \sum_{y \in W} (-v)^{\ell(w)-\ell(y)} \cdot P_{y,w}(q^{-1}) \cdot v^{-\ell(y)} \cdot t_y.\]
 This element is just $(-1)^{\ell(w)} \cdot \iota(c_w)$, where $\iota$ is the $\cA$-algebra automorphism of $\cH_q$ with $t_w \mapsto (-q)^{\ell(w)} \cdot t_{w^{-1}}^{-1}$ for 
$w \in W$. (When checking this, it helps to recall $\overline {c_w } =c_w$.) This observation transforms Dyer's results into what is stated here.
\end{remark}


Moving on to the analogous decomposition of $C_x A_y$, we have this lemma. Recall from Theorem \ref{mult-thm} that if $s \in S$ then $C_s = q^{-1} (T_s + 1) \in \cH_{q^2}$.

\begin{lemma}\label{universal-structure-lem} Assume $(W,S)$ is a universal Coxeter system. Suppose $s \in S$ and $w \in \I$. 
\ben

\item[(a)] If $s \in \Des(w)$ then $C_s A_w = \(q+q^{-1}\) A_w$.
\item[(b)] If $s \notin \Des(w)$ then 
\[ C_s A_w = \begin{cases} A_{sws^*} + A_{r  w r^*}&\text{if $\Des(w) = \{r\}$ and $\Des(r wr^*) = \{s\}$}
\\
A_{sws^*} + A_s &\text{if $w \in S$ and $s=s^*$}
\\
(v+v^{-1}) A_s &\text{if $w=1$ and $s=s^*$}
\\
A_{sws^*}&\text{otherwise}.
\end{cases}
\]
\een
\end{lemma}

\begin{proof}
%
Part (a) is immediate from Theorem \ref{mult-thm}. 
If $w =1$ then $\msig(y\xrightarrow{s}w)=\msig(y\xrightarrow{s}1) = 0$ for all $y \in \I$ with $sy<y$ so by  Theorem \ref{mult-thm} we have 
$C_s A_1 = (v+v^{-1})^c A_{s\act 1}$ where $c = \delta_{s,s^*}$. This proves part (b) when $w=1$. 

For the remaining cases, assume $w \neq 1$ and  $\Des(w) = \{ r\} \neq \{s\}$.
Combining Theorem \ref{mult-thm} and Lemma \ref{msig-lem} gives 
$C_s A_w = A_{s\act 1} + \sum_{y \in X} A_y$, where  $X \subset \I$ is the subset which contains $s$ if $s =s^*$ and $w=r \in S$, and which contains $rwr^*$ if $rwr^* \in \I$ and $\Des(rwr^*) = \{s\}$.
Since  $rwr^*$  always belongs to $\I$ and since $\Des(rwr^*) = \{s\}$ implies $w\notin S$, the set $X$ contains at most one element and our formula $C_s A_w = A_{s\act 1} + \sum_{y \in X} A_y$ reduces to the cases in the lemma.
\end{proof}

We now make this definition, after Definition \ref{dyerdef}. 

\begin{definition}\label{afterdyerdef} 
Assume $(W,S)$ is a universal Coxeter system. Let $w \in \I$ and $n = \rho(w)$, and  suppose   $s_i \in S$ such that $w =s_1 \act s_2 \act \cdots \act s_n\act 1$.
For each integer $j \in \ZZ$,   define $A(w,j) \in \cM_{q^2}$ recursively according to the following cases:


\ben
\item[(a)]  If $2 \leq j \leq n-1$ (so that $n\geq 3$) and $s_{j-1} = s_{j+1}$, then  set 
\[ A(w,j) = A_{w'} + A(w',j-1),\qquad\text{where }
w' = s_1\act  \cdots \act  \widehat{s}_j \act  \widehat{ s}_{j+1} \act \cdots \act s_n.\]
Here, we again write $\widehat{s}_j$ to indicate that the factor $s_j$ is omitted.

\item[(b)] If $j = n$ and $n\geq 2$ and $\{s_{n-1},s_n\} \subset \I$, then  set
\[ A(w,j) = A_{w'} + A(w',n-1),\qquad\text{where } w' = s_1\act \cdots\act s_{n-1}.\]

\item[(c)] Otherwise  set $A(w,j) = 0$.

\een

\end{definition}

%
%

Using this notation, the following analog of Theorem \ref{c-structure} now decomposes the product $C_x A_y$ in terms of the distinguished basis $\(A_z\)_{z \in \I}$ of $\cM_{q^2}$. This result shows that the Laurent polynomials $\( h^\sigma_{x,y;z}\)_{x \in W, y,z \in \i}$ have nonnegative coefficients, but in contrast to our previous situation, $h^\sigma_{x,y;z}$ does not typically have nonnegative coefficients when written as a  polynomial in $v+v^{-1}$.

\begin{theorem}\label{A-structure} Assume $(W,S)$ is  universal. 
If $x \in W$ and $y  \in \I$  and  $n =\ell(x)$, then
\[ C_x A_y = \begin{cases} 
 (v+v^{-1}) \(A_{x \act 1} + A(x \act 1,n)\)&\text{if $x \neq 1$ and $y =1$ and
$ \mathrm{Des}_R(x) \subset \I $}
\\
(q+q^{-1}) \( A_{xs\act y} + A(xs\act y,n)\)&\text{if $\mathrm{Des}_R(x) = \Des(y) = \{s\} \neq \varnothing$}
\\
A_{x \act y} + A(x \act y,n) + A(x\act y,n+1)&\text{otherwise.}

\end{cases}
\]

%
%
%
%
%
%
%
\end{theorem}


\begin{proof}
The proof is similar to that of \cite[Theorem 3.12]{Dyer}, and proceeds by induction on $n$. If $n \in \{0,1\}$ then the theorem reduces to Lemma \ref{universal-structure-lem} (checking this fact is a healthy exercise which we leave to the reader), so we may assume $\ell(x) \geq 2$ and that 
\[ x=x'rs\qquad\text{for some $x' \in W$ and $r,s \in S$ with $\ell(x') = \ell(x)-2$}.\]
It follows from Theorem \ref{c-structure} (noting that the $\ZZ$-linear map $\cH_{q} \to \cH_{q^2}$ with $v^n \mapsto q^n$ and $t_w \mapsto T_w$ is a ring embedding with $c_w \mapsto C_w$) that 
\be\label{noted} C_x = \begin{cases} C_{xs} C_s - C_{x'} &\text{if $\mathrm{Des}_R(x') = \{s\}$} \\ 
C_{xs} C_s &\text{otherwise.}
\end{cases}\ee
It suffices to consider the following five cases, exactly one of which must occur:
\ben
\item[(i)] Suppose $y=1$.  
Then $A(x\act y,n+1) = 0$ and so we  wish to show that
$
C_x A_1 =  (v+v^{-1})^c\cdot \( A_{x\act 1} + A(x\act 1,n)\)
$ 
where $c = | \{s\} \cap \I|$.

\item[(ii)]  Suppose $s \in \Des(y)$. We then  wish to show that
$
C_x A_y = (q+q^{-1}) \( A_{xs\act y} + A(xs\act y,n)\).
$

\item[(iii)] Suppose $y \in S$ and $ s\notin \Des(y)$ and $s =s^*$.  Then $A(x\act y,n+1) = A_{x\act 1} + A(x\act1 ,n)$  and so we wish to show  
$
 C_x A_y = A_{x \act y} + A(x\act y,n) + A_{x\act 1} + A(x\act1 ,n).
 $

\item[(iv)] Suppose $\rho(y) = 1$ and $ s\notin \Des(y)$ but either $y \notin S$ or $s\neq s^*$.  Then $A(x\act y,n+1) = 0$ and so we wish to show  
$
C_x A_y = A_{x\act y} + A(x\act y,n).
$

\item[(v)] Suppose $\rho(y) \geq 2$ and $ s\notin \Des(y)$.
We then want
$
 C_x A_y = A_{x\act y} + A(x\act y,n) + A(x\act y,n+1).
$
\een
The proof of each case is similar, and involves substituting \eqref{noted} for $C_x$ and then applying Lemma \ref{universal-structure-lem} and induction.
 Case (v) is the most complicated, but its proof is nearly the same as that 
of \cite[Lemma 6.2]{Dyer}. We demonstrate (i)  as an example and leave the rest to the reader.

For case (i), suppose $y=1$ and let $c = | \{s\} \cap \I|$; recall that $\mathrm{Des}_R(x) = \{s\}$ by assumption.
If $\mathrm{Des}_R(x') \neq \{s\}$ then $C_x = C_{x'r} C_s$ by \eqref{noted} and $A(x\act 1,n-1)=0$, in which case by  Lemma \ref{universal-structure-lem} and then induction  we get
\[ \ba C_x A_1 
&=C_{x'r} C_s A_1
\\&
= (v+v^{-1})^c\cdot  C_{x'r} A_{s \act 1}
\\&= (v+v^{-1})^c\cdot ( \A_{x \act 1} + \underbrace{A(x\act 1,n-1)}_{=0}+ A(x\act 1,n)),\ea\]
which is what we want to show. 
Alternatively, if $\mathrm{Des}_R(x') = \{s\}$ then $C_x = C_{x'r} C_s-C_{x'}$ by \eqref{noted} and $A(x\act 1,n-1) =A_{x'\act 1} + A(x'\act 1,n-2)$, so by induction
$C_{x'} A_1 = (v+v^{-1})^c \cdot A(x\act 1,n-1)$. In this case  by Lemma \ref{universal-structure-lem} and then induction we have
\[ \ba C_x A_1 
&=
( C_{x'r} C_s-C_{x'})A_1
\\&
= (v+v^{-1})^c\cdot  C_{x'r} A_{s \act 1} - C_{x'} A_1
\\&= (v+v^{-1})^c\cdot  \(A_{x \act 1} +  A(x\act 1,n)\)  + \underbrace{(v+v^{-1})^c\cdot A(x\act 1,n-1)-  C_{x'} A_1}_{=0}
\ea\]
which is again  the desired formula.
\end{proof}

Wrapping up, we have this corollary immediately from Theorems \ref{c-structure} and \ref{A-structure}.

\begin{corollary}\label{A-cor} If $(W,S)$ is a universal Coxeter system then each of the families \[\(h_{x,y;z}\)_{x,y,z \in W}\qquand (\wt h_{x,y;z})_{x,y,z \in W}\qquand (h^\sigma_{x,y;z})_{x \in W,\hs y,z\in \I} \] consists of Laurent polynomials in $\cA = \ZZ[v,v^{-1}]$ with nonnegative  coefficients.
\end{corollary}

\subsection{Proof of the positivity conjecture for  universal structure constants
}\label{last-sect}

As previously, $(W,S)$ is a universal Coxeter system with a fixed $S$-preserving involution $* \in \Aut(W)$. 
We devote this final section to proving  Conjecture \descref{C$'$} for universal Coxeter systems$-$i.e., that the  Laurent polynomials $h^\pm_{x,y;z}  = \tfrac{1}{2} \( \wt h_{x,y;z} \pm h^\sigma_{x,y;z} \)$ defined in Section \ref{conjecture-sect} always have nonnegative coefficients.


To begin, it is useful to recall the following  notation from the proof of Proposition \ref{h-parity}.
Given $w \in W$, let $w^\dag = {w^*}^{-1}$ and more generally
let 
$h \mapsto h^\dag$ denote the $\cA$-linear map $\cH_q \to \cH_q$ with $(t_w)^\dag = t_{w^\dag}$ for $w \in W$. Observe that $\dag$ is an anti-automorphism (of $\cA$-algebras) and that  $(c_w)^\dag = c_{w^\dag}$ for all $w \in W$ by Lemma \ref{kl-cor}.
We now state two technical lemmas associated with 
Definitions \ref{dyerdef} and \ref{afterdyerdef}.

\begin{lemma}\label{a-lem} Assume $(W,S)$ is a universal Coxeter system.
Suppose $u,t \in W$ such that $\ell(u\act t) = 2\ell(u) + \ell(t)$ and $\ell(t) \in \{0,1\}$ and $t=t^*$. Fix an integer $n \leq \ell(u)$.
Then there exists a unique integer $k\geq 0$ and a unique  sequence of elements \[u=u_0> u_1 > \dots > u_k\] in $W$ (descending with respect to the Bruhat order), 
such that  
$c(ut,n) = \sum_{i=1}^k c_{u_it}$. 
This sequence  has the following additional properties:
\ben
\item[(a)] For each $0\leq i \leq k$ we have $\ell(u_i \act t) = 2\ell(u_i) + \ell(t)$.
\item[(b)] $c(utw,n) = \sum_{i=1}^k c_{u_itw} + c(u_ktw,n-k) $ for any  $w \in W$ with $\ell(utw) = \ell(u) + \ell(t) + \ell(w)$.

\item[(c)] $ A(u\act t,n) = \sum_{i=1}^k A_{u_i\act t} + \delta\cdot A(u_k\act t,n-k)$ where  $\delta = \begin{cases} 1& \text{if $n-k =\ell(u_k)+1$} \\ 0 &\text{otherwise}.\end{cases}$ 

\een

%
%
%
%

\end{lemma}

\begin{remark} Note that we may  have $k=0$ in this lemma; this indicates that $c(ut,n) = 0$. In this case the sums $\sum_{k=1}^n$ are considered to be zero, and we automatically have $\delta = 0$ in part (c) since $n < \ell(u_0)+1$ by hypothesis. 
\end{remark}

\begin{proof}
%

We  sketch the proof of this lemma, as everything derives from the definitions  in a straightforward way  by induction on $\ell(u)$.
The existence of the sequence of elements $u=u_0> u_1 > \dots > u_k$ follows from Definition \ref{dyerdef}  by inspection,
as does property (a).
Property (b) holds because 
the first $k+1$ terms in the expansion of $c(utw,n)$, which one gets by applying Definition \ref{dyerdef} successively, depend only on the first $n+k$ factors in the unique reduced expression for $utw$.
%
Part (c) follows from the fact that the same sequence of elements in $S$ gives both the unique reduced expression for $ut$ and the unique reduced $\I$-expression for $u\act t$. Noting this and comparing Definitions \ref{dyerdef} and \ref{afterdyerdef} (while remembering $n \leq \ell(u)$), we deduce that $A(u\act t,n) = \sum_{i=1}^k A_{u_i\act t} + A(u_k\act t,n-k)$, and that $A(u_k \act t,n-k)$ is zero unless $n-k =\rho(u_k \act t)$. The latter condition is equivalent to having both $\ell(t) = 1$ and $n-k = \ell(u_k) + 1$; however, if $\ell(t) = 0$ while $n-k = \ell(u_k) + 1$ then $A(u_k \act t,n-k)$ is  zero by definition.
\end{proof}

In what follows, we let $\Phi : \cM_{q^2} \to \cH_q$ denote the $\cA$-linear map with $A_w \mapsto c_w$ for $w \in \I$.


\begin{lemma}\label{b-lem} Assume $(W,S)$ is a universal Coxeter system.
Suppose $x \in W$ and $ s\in S \cap \I$ such that $s\notin \mathrm{Des}_R(x)$. If $n = \ell(x)$, then 
\[ c(x\act s,n+1) = \Phi\( A(x \act s,n+1) \).\]
\end{lemma}

\begin{proof}
If $x = 1$ or if $\mathrm{Des}_R(x) \not \subset \I$ then the lemma holds since $c(x\act s,n+1) $ and $A(x\act s,n+1) $ are both zero. Assume $x \neq 1$ so that $x=x'r$ for some $y \in W$ and $r\in S \cap \I$ with $\ell(x') = \ell(x) - 1$.
Then 
$ c(x\act s,n+1) = c_{x' \act r} + c(x' \act r, n)$ and $A(x'\act s,n+1) = A_{x' \act r} + A(x'\act r,n)$,
so the lemma follows by induction on $n$.
\end{proof}



We may now state our final result, which establishes Conjecture \descref{C$'$}  in the universal case.

\begin{theorem}\label{last-thm} If $(W,S)$ is a universal Coxeter system and $* \in \Aut(W)$ is any $S$-preserving involution, then the Laurent polynomials $h^\pm_{x,y,z}$ defined by  \eqref{hpm-def} have nonnegative integer coefficients for all $x\in W$ and $y,z \in \I$.  
\end{theorem}

\begin{proof}
Let $\cH_q^+=\NN[v,v^{-1}]\spanning\{c_w : w \in W\}$ denote the set of elements in $\cH_q$ whose coefficients with respect to the  Kazhdan-Lusztig basis $\( c_w\)_{w \in W}$ have nonnegative coefficients. Note that $\cH_q^+$ is preserved by $\dag$ since $(c_w)^\dag = c_{w^\dag}$.

Let $x \in W$ and $y \in \I$. By Theorems \ref{c-structure} and \ref{A-structure} we know that 
$c_x c_y c_{x^\dag} \in \cH_q^+$ and $\Phi(C_x A_y) \in \cH_q^+$, and if we write
$c_x c_y c_{x^\dag} \pm  \Phi(C_x A_y)
= \sum_{z \in W} p^\pm_z c_z$ for some polynomials $p^\pm_z \in  \ZZ[v,v^{-1}]$,
 then by definition $h_{x,y;z}^\pm = \frac{1}{2} p^\pm_z $ for each $z \in \I$. It is thus immediate that every $h^+_{x,y;z}$ has nonnegative coefficients, and to prove the theorem it is  enough to show that 
\be \label{want9.8} 
c_x c_y c_{{x^*}^{-1}}  - \Phi(C_x A_y) \in \cH_q^+.
\ee
To this end, let $n = \ell(x)$. If $n=0$ then \eqref{want9.8} automatically holds since the left hand side is zero, so we may assume $n\geq 1$. 
There are three cases, which we consider in turn:
\ben
\item[(a)] Suppose $y = 1$.  Expand the products 
 $c_x c_y c_{x^\dag} = c_x c_{x^\dag}$ and $C_x A_y = C_x A_1$ according to  Theorems \ref{c-structure} and \ref{A-structure}. These expansions take one of two forms according to 
whether $s=s^*$, and applying Lemma \ref{b-lem} to the  terms in either case shows that \eqref{want9.8} holds.
%

\item[(b)] Suppose $y \neq 1$ and $ \Des(y) 
\neq  \mathrm{Des}_R(x)$. Apply Theorem \ref{c-structure}    to expand the product $c_x c_y c_{x^\dag}$, by  expanding first 
 $c_x c_y$ and then  $(c_x c_y) c_{x^\dag}$. There are again two cases according to whether $y \in S$. On comparing the resulting terms to Theorem \ref{A-structure} (while  noting Lemma \ref{b-lem}), one finds that 
 \eqref{want9.8} will hold if we can prove the following claims:
\ben
\item[(b1)] If $\ell(y) \geq1$ then we have $c(xy,n) c_{x^\dag}  - \Phi\( A(x\act y,n)\) \in \cH_q^+$.
\item[(b2)] If $\ell(y) \geq 2$ then we have $c(xy,n+1) c_{x^\dag} - \Phi\( A(x\act y,n+1)\) \in \cH_q^+$.
\een
To prove (b1), 
 write $y = ztz^\dag$ where $z,t \in W$ such that $\ell(t) \leq 1$ and $t=t^*$ and $\ell(ztz^\dag) = 2\ell(z) + \ell(t)$.
Now let $u=xz$ and let $u=u_0 > u_1 >\dots>u_k$ be the corresponding sequence of elements in $W$ described in Lemma \ref{a-lem}, so that $c(ut,n) = \sum_{i=1}^k c_{u_kt}$.
%
Using part (b) of Lemma \ref{a-lem} and the fact that $\dag$ is an anti-automorphism, we then have
\[\ba  c(xy,n) c_{x^\dag} = c(utz^\dag,n) c_{x^\dag} &=
 \( \sum_{i=1}^k  c_{u_it z^\dag} +    c(u_k tz^\dag,n-k)  \) c_{x^\dag} 
\\&
= \( c_x  \sum_{i=1}^k  c_{zt(u_i)^\dag}  \)^\dag + c(u_k tz^\dag,n-k)  c_{x^\dag} 
\\&
=   \( \sum_{i=1}^k c(ut(u_i)^\dag,n)  \)^\dag + \(\text{an element of }\cH_q^+\)
.
\ea\]
Here, the last equality follows by applying Theorem \ref{c-structure} to the terms in the sum on the second line.
Since each $c(ut(u_i)^\dag,n) = \sum_{j=1}^k c_{u_j t (u_i)^\dag} + \(\text{an element of }\cH_q^+\)$ by parts (a) and (b) of Lemma \ref{a-lem},  after collecting  terms in $\cH_q^+$ we get
\[
c(xy,n) c_{x^\dag} = \sum_{i=1}^k c_{u_i \act t } + c(u_k\act t,n-k) + \(\text{an element of }\cH_q^+\)
.
\]
By part (c) of Lemma \ref{a-lem}, however, we have
$A(x\act y,n) = \sum_{i=1}^k A_{u_i \act t} + \delta\cdot A(u_k \act t,n-k),$
where $\delta \in \{0,1\}$  is zero unless  $n-k = \ell(x_k) + 1$. If $\delta=1$ then $A( u_k\act t,n-k) = 0$ unless $t \in S\cap \I$, and so 
 (b1) follows by Lemma \ref{b-lem}.

One proves (b2) by replacing $n$ with $n+1$ in the preceding argument. Our applications of Lemma \ref{a-lem} remain valid after this substitution because we assume $\ell(y) \geq 2$, which implies $\ell(u) \geq 1$ and in turn  $n+1 \leq \ell(u)$.

\item[(c)] Suppose $y \neq 1$ and $ \mathrm{Des}_R(x) = \Des(y) = \{s\} $ for some $s \in S$. 
Using  Theorems \ref{c-structure} and \ref{A-structure}   to 
expand  the products $(c_x c_y) c_{x^\dag}$ and $C_x A_y$ gives
\[ 
c_xc_y c_{x^\dag} = (q+2+q^{-1}) \cdot c_{xs \act y} + (q+2+q^{-1})\cdot c(xs\act y,m) + (v+v^{-1}) \cdot c(xsy,n)c_{x^\dag}
\]
where $m = \ell(xsy) = n+\ell(y)-1$ and
\[
C_xA_y = (q+q^{-1}) \cdot A_{xs\act y} + (q+q^{-1}) \cdot A(xs\act y,n).
\]
 Note that   $c(xs\act y,m) \in \cH_q^+$ and $c(xsy,n)c_{x^\dag} \in \cH_q^+$ automatically.

We have two cases to consider: either $y = s \in \I$ or $\ell(y)\geq 2$. In the former case $m=n$, and so it follows by Lemma \ref{b-lem} that
 \[c(xs \act y,m) = c(xs\act y,n)  = \Phi\(A(xs\act y,n)\)\]  and therefore \eqref{want9.8} holds.
To deal with the remaining case, assume $\ell(y) \geq 2$. To prove \eqref{want9.8}
it then suffices  to show that
\be\label{3want}  (v+v^{-1})\cdot  c(xsy,n) c_{x^\dag} = (q+q^{-1}) \cdot \Phi\(  A(xs\act y,n) \) + \(\text{an element of }\cH_q^+\).\ee
The proof of this identity is similar to the arguments in part (b). A sketch goes as follows. First write $y = ztz^\dag$ where $z,t \in W$ such that $\ell(t) \leq 1 \leq \ell(z)$ and $t^*=t$ and $\ell(ztz^\dag) = 2\ell(z) + \ell(t)$.
Let $u=xsz$ and let $u=u_0 > u_1 >\dots >u_k$ be the  sequence of elements in $W$ afforded by Lemma \ref{a-lem}, so that $c(ut,n) = \sum_{i=1}^k c_{u_kt}$.
By now rewriting $c(xsy,n) = c(utz^\dag,n)$ in terms of the elements $u_i$ and expanding various products using the properties in Lemma \ref{a-lem}, one obtains 
\[ c(xsy,n) c_{x^\dag} = (v+v^{-1}) \(\sum_{i=1}^k c_{u_i \act t} + c(u_k \act t,n-k)\) +  \(\text{an element of }\cH_q^+\).
\]
Comparing this to the formula for $A(xs\act y,n)$ in part (c) of Lemma \ref{a-lem} then shows that \eqref{3want} holds, as a consequence of  Lemma \ref{b-lem}.
%
\een
\end{proof}

\end{document}